\documentclass[12pt]{article}
\usepackage[margin=1in]{geometry}  
\usepackage{graphicx}              
\usepackage{amsmath}               
\usepackage{amsthm}
\usepackage{amssymb}
\usepackage[ansinew]{inputenc}
\usepackage{array}
\usepackage{amsxtra}
\usepackage{amstext}
\usepackage{latexsym}
\usepackage{dsfont}
\usepackage[all]{xy}
\usepackage{amsfonts}
\usepackage{eucal}
\usepackage{amscd}
\usepackage{multicol}
\usepackage[all]{xy}           
\usepackage{graphicx}
\usepackage{color}
\usepackage{colordvi}
\usepackage{xspace}

\usepackage{enumitem}  
\usepackage{calc}
\setlist{labelindent=1pt,itemsep=.5em}
\setlist[itemize]{leftmargin=1.2cm}
\setlist[enumerate]{itemindent=0em,leftmargin=1.2cm}
\setlist[enumerate,1]{label={\upshape(\roman*)}}

\usepackage{ifpdf}
\ifpdf
 \usepackage[colorlinks,final,hyperindex]{hyperref}
\else
\usepackage[colorlinks,final,hyperindex,hypertex]{hyperref}
\fi

\usepackage{cite}

\usepackage{authblk}

\allowdisplaybreaks

\makeatletter
\newcommand{\subjclass}[2][2020]{%
  \let\@oldtitle\@title%
  \gdef\@title{\@oldtitle\footnotetext{#1 \emph{Mathematics subject classification}: #2}}%
}
\newcommand{\keywords}[1]{%
  \let\@@oldtitle\@title%
  \gdef\@title{\@@oldtitle\footnotetext{\emph{Keywords}: #1.}}%
}
\makeatother

	\newtheorem{thm}{Theorem}[section]
	
	\newtheorem{cor}[thm]{Corollary}
	\newtheorem{pro}[thm]{Proposition}
	\newtheorem{defi}[thm]{Definition}
	\newtheorem{ex}[thm]{Example}
	\newtheorem{rmk}[thm]{Remark}

\newcommand{\A}{\rm A}

\newcommand{\mrm}[1]{{\rm #1}}
\newcommand{\End}{\mrm{End}} 
\newcommand{\id}{\mrm{id}} 

\title{Nearly associative and nearly Hom-associative algebras and bialgebras}
\author{Mafoya Landry Dassoundo$^{1}$,
Sergei Silvestrov$^{2}$ \\
\small{$^{1}$Chern Institute of Mathematics and LPMC,
Nankai University, Tianjin 300071, China \authorcr
e-mail: dassoundo@yahoo.com \authorcr
$^{2}$ Division of Mathematics and Physics,
School of Education, Culture and Communication, \authorcr
M\"{a}lardalen University, Box 883, 72123 V\"{a}steras, Sweden. \authorcr
e-mail: sergei.silvestrov@mdh.se}}


\date{}

\begin{document}

\maketitle

\abstract{Basic definitions and properties of nearly associative algebras are described.
Nearly associative algebras are proved to be Lie-admissible algebras.
Two-dimensional nearly associative algebras are classified, and its main classes are derived.
The bimodules, matched pairs and Manin triple  of a nearly associative
algebras are derived and their equivalence with nearly associative bialgebras is proved.
Basic definitions and properties of nearly Hom-associative  algebras are described.
Related bimodules and matched pairs are given, and associated identities are established.}

\footnote[0]{{\it Corresponding author}: Sergei Silvestrov, sergei.silvestrov@mdh.se}


\section{Introduction}
\label{sec:intro}
An algebra $A$ with a bilinear
product $\cdot: A\times A\rightarrow A$ is not necessarily associative or possibly non-associative if
possibly there exist $x,y,z\in A$ such that $(x\cdot y)\cdot z-x\cdot(y\cdot z)\neq 0.$
If such $x,y,z\in A$ exist, then algebra is not associative.
The term non-associative algebras is used often to mean all possibly non-associative algebras, including also the associative algebras. Associative algebras, Lie algebras, and Jordan algebras are
well-known sub-classes of non-associative algebras in the sense of possibly not associative algebras
\cite{DassoundoSilvestrov:nearlyhomass:Schafer:intrononasal}.

Hom-algebraic structures originated from  quasi-deformations of Lie algebras of
vector fields which gave rise to quasi-Lie algebras, defined as  generalized
Lie structures in which the skew-symmetry and Jacobi conditions are twisted.
Hom-Lie algebras and more general quasi-Hom-Lie algebras where introduced first by Silvestrov and his students Hartwig and Larsson in \cite{DassoundoSilvestrov:nearlyhomass:HaLaSil}, where the general quasi-deformations and discretizations of Lie algebras of vector fields using general twisted derivations, $\sigma$-derivations, and a general method for construction of deformations of Witt and Virasoro type algebras based on twisted derivations have been developed. The initial motivation came from examples of $q$-deformed Jacobi identities discovered in $q$-deformed versions and other discrete modifications of differential calculi and homological algebra, $q$-deformed Lie algebras and other algebras important in string theory, vertex models in conformal field theory, quantum mechanics and quantum field theory, such as the $q$-deformed Heisenberg algebras, $q$-deformed oscillator algebras, $q$-deformed Witt, $q$-deformed Virasoro algebras and related $q$-deformations of infinite-dimensional algebras  \cite{DassoundoSilvestrov:nearlyhomass:AizawaSaito,DassoundoSilvestrov:nearlyhomass:ChaiElinPop,DassoundoSilvestrov:nearlyhomass:ChaiIsLukPopPresn,DassoundoSilvestrov:nearlyhomass:ChaiKuLuk,DassoundoSilvestrov:nearlyhomass:ChaiPopPres,DassoundoSilvestrov:nearlyhomass:CurtrZachos1,DassoundoSilvestrov:nearlyhomass:DamKu,DassoundoSilvestrov:nearlyhomass:DaskaloyannisGendefVir,DassoundoSilvestrov:nearlyhomass:Hu,DassoundoSilvestrov:nearlyhomass:Kassel92,
DassoundoSilvestrov:nearlyhomass:LiuKQuantumCentExt,DassoundoSilvestrov:nearlyhomass:LiuKQCharQuantWittAlg,DassoundoSilvestrov:nearlyhomass:LiuKQPhDthesis}.

Possibility of studying, within the same framework, $q$-deformations of Lie algebras and such well-known generalizations of Lie algebras as the color and super Lie algebras provided further general motivation for development of quasi-Lie algebras and subclasses of quasi-Hom-Lie algebras and Hom-Lie algebras. The general abstract quasi-Lie algebras and the subclasses of quasi-Hom-Lie algebras and Hom-Lie algebras, as well as their color (graded) counterparts, color (graded) quasi-Lie algebras, color (graded) quasi-Hom-Lie algebras and color (graded) Hom-Lie algebras, including in particular the super quasi-Lie algebras, super quasi-Hom-Lie algebras, and super Hom-Lie algebras, have been introduced in
\cite{DassoundoSilvestrov:nearlyhomass:HaLaSil,
DassoundoSilvestrov:nearlyhomass:LarssonSilvJA2005:QuasiHomLieCentExt2cocyid,
DassoundoSilvestrov:nearlyhomass:LarssonSilvCM2005:QuasiLieAl,
DassoundoSilvestrov:nearlyhomass:LSGradedquasiLiealg,
DassoundoSilvestrov:nearlyhomass:SigSilGrquasiLieWitttype:CzJPhys2006,DassoundoSilvestrov:nearlyhomass:SigSilv:LiecolHomLieWittcex:GLTbdSpr2009}.
In \cite{DassoundoSilvestrov:nearlyhomass:MakhloufSilvestrov:homstructure}, Hom-associative algebras were introduced, generalizing associative algebras by twisting the associativity law by a linear map. Hom-associative algebra is a triple $(A, \cdot, \alpha)$ consisting of
a linear space $A$, a bilinear product $\cdot:A\times A\rightarrow A$
and a linear map $\alpha: A\rightarrow A$, satisfying
$a_{\alpha,\cdot}(x,y,z)=(x\cdot y)\cdot \alpha(z)-\alpha(x)\cdot(y\cdot z)=0,$
for any $x,y,z\in A$.
In \cite{DassoundoSilvestrov:nearlyhomass:MakhloufSilvestrov:homstructure}, alongside Hom-associative algebras, the Hom-Lie admissible algebras generalizing Lie-admissible algebras, were introduced as Hom-algebras such that the commutator product, defined using the multiplication in a Hom-algebra, yields a Hom-Lie algebra, and also Hom-associative algebras were shown to be Hom-Lie admissible. Moreover, in \cite{DassoundoSilvestrov:nearlyhomass:MakhloufSilvestrov:homstructure}, more general $G$-Hom-associative algebras including Hom-associative algebras, Hom-Vinberg algebras (Hom-left symmetric algebras), Hom-pre-Lie algebras (Hom-right symmetric algebras), and some other Hom-algebra structures, generalizing $G$-associative algebras, Vinberg and pre-Lie algebras respectively, have been introduced
and shown to be Hom-Lie admissible, meaning that for these classes of Hom-algebras, the operation of taking commutator leads to Hom-Lie algebras as well. Also, flexible Hom-algebras have been introduced, connections to Hom-algebra generalizations of derivations and of adjoint maps have been noticed, and some low-dimensional Hom-Lie algebras have been described.
The enveloping algebras of Hom-Lie algebras were considered in \cite{DassoundoSilvestrov:nearlyhomass:Yau:HomEnv} using combinatorial objects of weighted binary trees. In \cite{DassoundoSilvestrov:nearlyhomass:HeMaSiUnAlHomAss}, for Hom-associative algebras and Hom-Lie algebras, the envelopment problem, operads, and the Diamond Lemma and Hilbert series for the Hom-associative operad and free algebra have been studied. Strong Hom-associativity yielding a confluent rewrite system and a basis for the free strongly hom-associative algebra has been considered in \cite{DassoundoSilvestrov:nearlyhomass:He:stronghomassociativity}. An explicit constructive way, based on free Hom-associative algebras with involutive twisting, was developed in \cite{DassoundoSilvestrov:nearlyhomass:GuoZhZheUEPBWHLieA} to obtain the universal enveloping algebras and Poincar{\'e}-Birkhoff-Witt type theorem for Hom-Lie algebras with involutive twisting map. Free
Hom-associative color algebra on a Hom-module and enveloping algebra of color Hom-Lie algebras with involutive twisting and also with more general conditions on the powers of twisting map was constructed, and Poincar{\'e}-Birkhoff-Witt type theorem was obtained in \cite{DassoundoSilvestrov:nearlyhomass:ArmakanSilvFarh:envalcolhomLieal,DassoundoSilvestrov:nearlyhomass:ArmakanSilvFarh:envalcerttypecolhomLieal}.
It is worth noticing here that, in the subclass of Hom-Lie algebras, the  skew-symmetry is untwisted, whereas the Jacobi identity is twisted by a single linear map and contains three terms as in Lie algebras, reducing to ordinary Lie algebras when the twisting linear map is the identity map.

Hom-algebra structures include their classical counterparts and open new broad possibilities for deformations, extensions to Hom-algebra structures of representations, homology, cohomology and formal deformations, Hom-modules and hom-bimodules, Hom-Lie admissible Hom-coalgebras, Hom-coalgebras, Hom-bialgebras, Hom-Hopf algebras, $L$-modules, $L$-comodules and Hom-Lie quasi-bialgebras, $n$-ary generalizations of biHom-Lie algebras and biHom-associative algebras and generalized derivations, Rota-Baxter operators, Hom-dendriform color algebras, Rota-Baxter bisystems and covariant bialgebras, Rota-Baxter cosystems, coquasitriangular mixed bialgebras, coassociative Yang-Baxter pairs, coassociative Yang-Baxter equation and generalizations of Rota-Baxter systems and algebras, curved $\mathcal{O}$-operator systems and their connections with tridendriform systems and pre-Lie algebras, BiHom-algebras, BiHom-Frobenius algebras and double constructions, infinitesimal biHom-bialgebras and Hom-dendriform $D$-bialgebras, Hom-algebras has been considered from a category theory point of view \cite{DassoundoSilvestrov:nearlyhomass:AmmarEjbehiMakhlouf:homdeformation,
DassoundoSilvestrov:nearlyhomass:Bakayoko:LaplacehomLiequasibialg,
DassoundoSilvestrov:nearlyhomass:Bakayoko:LmodcomodhomLiequasibialg,
DassoundoSilvestrov:nearlyhomass:BakBan:bimodrotbaxt,
DassoundoSilvestrov:nearlyhomass:BakyokoSilvestrov:HomleftsymHomdendicolorYauTwi,
DassoundoSilvestrov:nearlyhomass:BakyokoSilvestrov:MultiplicnHomLiecoloralg,
DassoundoSilvestrov:nearlyhomass:BenMakh:Hombiliform,
DassoundoSilvestrov:nearlyhomass:BenAbdeljElhamdKaygorMakhl201920GenDernBiHomLiealg,
DassoundoSilvestrov:nearlyhomass:CaenGoyv:MonHomHopf,
DassoundoSilvestrov:nearlyhomass:GrMakMenPan:Bihom1,
DassoundoSilvestrov:nearlyhomass:HassanzadehShapiroSutlu:CyclichomolHomasal,
DassoundoSilvestrov:nearlyhomass:HounkonnouDassoundo:centersymalgbialg,
DassoundoSilvestrov:nearlyhomass:HounkonnouHoundedjiSilvestrov:DoubleconstrbiHomFrobalg,
DassoundoSilvestrov:nearlyhomass:kms:narygenBiHomLieBiHomassalgebras2020,
DassoundoSilvestrov:nearlyhomass:LarssonSigSilvJGLTA2008,
DassoundoSilvestrov:nearlyhomass:LarssonSilvJA2005:QuasiHomLieCentExt2cocyid,
DassoundoSilvestrov:nearlyhomass:LarssonSilvestrov:Quasidefsl2twder,
DassoundoSilvestrov:nearlyhomass:MaMakhSil:CurvedOoperatorSyst,
DassoundoSilvestrov:nearlyhomass:MaMakhSil:RotaBaxbisyscovbialg,
DassoundoSilvestrov:nearlyhomass:MaMakhSil:RotaBaxCosyCoquasitriMixBial,
DassoundoSilvestrov:nearlyhomass:MakhSil:HomHopf,DassoundoSilvestrov:nearlyhomass:MakhSilv:HomDeform,
DassoundoSilvestrov:nearlyhomass:MakhSilv:HomAlgHomCoalg,
DassoundoSilvestrov:nearlyhomass:MakhloufYau:RotaBaHomLieadm,
DassoundoSilvestrov:nearlyhomass:RichardSilvestrovJA2008,
DassoundoSilvestrov:nearlyhomass:RichardSilvestrovGLTbnd20092009,
DassoundoSilvestrov:nearlyhomass:ShengBai:homLiebialg,
DassoundoSilvestrov:nearlyhomass:Sheng:homrep,
DassoundoSilvestrov:nearlyhomass:SilvestrovParadigmQLieQhomLie2007,
DassoundoSilvestrov:nearlyhomass:Yau:ModuleHomalg,
DassoundoSilvestrov:nearlyhomass:Yau:HomEnv,
DassoundoSilvestrov:nearlyhomass:Yau:HomHom,
DassoundoSilvestrov:nearlyhomass:Yau:HombialgcomoduleHomalg,
DassoundoSilvestrov:nearlyhomass:Yau:HomYangBaHomLiequasitribial}.

This paper is organized as follows. In Section~\ref{section1},
basic definitions and  fundamental identities and some elementary examples
of nearly associative algebras are given.
In Section~\ref{section2}, we derive the classification of the two-dimensional
nearly associative algebras and main classes are provided.
In Section~\ref{section3},
bimodules, duals bimodules  and matched pair of nearly associative algebras
are established and related identities are derived and
proved. In Section~\ref{section4}, Manin triple of nearly associative algebras
is given and its equivalence to the nearly associative bialgebras is derived.
In Section \ref{sec:homLieadmGHomass}, Hom-Lie-admissible, $G$-Hom-associative, flexible Hom-algebras, the result on Lie-admissibility of $G$-Hom-admissible algebras and subclasses of $G$-Hom-admissible algebras are reviewed.
In Section~\ref{sec:nearlyhomass}, main definitions and fundamental identities of
Hom-nearly associative algebras are given. Furthermore, the bimodules,
and matched pair of the Hom-nearly associative algebras are derived and related
properties are obtained.

\section{Nearly associative algebras: basic definitions and properties}\label{section1}

Throughout this paper, for simplicity of exposition, all linear spaces are assumed to be over field
$\mathbb{K}$ of characteristic is $0$, even though many results hold in general for other fields as well unchanged or with minor modifications.
An algebra is a couple $(A, \mu)$ consisting of a linear space $ A$ and a bilinear product $\mu: A\times A \rightarrow A$.
\begin{defi}
An algebra $(A, \cdot)$ is called nearly associative if, for all $x, y, z\in  A$,
\begin{eqnarray}\label{eq_identity}
x\cdot (y\cdot z)=  (z\cdot x)\cdot y.
\end{eqnarray}
\end{defi}
\begin{ex}
	Consider a two-dimensional  linear space $ A$ with basis $\{ e_1, e_2\}$.
	\begin{itemize}
		\item
		Then, $( A, \cdot)$ is a nearly associative  algebra, where $e_1\cdot e_1= e_1+ e_2$ and
		for all $(i,j)\neq (1,1)$ with $i, j\in \{1, 2\}$, $e_i\cdot e_j=0$.
		\item
		The linear product defined on $ A$ by: $e_1\cdot e_1= e_2$,
		$e_1\cdot e_2= e_1=e_2\cdot e_1$ and
		$e_2\cdot e_2=e_2$,  is such that $(A, \cdot)$ is a nearly associative  algebra.
	\end{itemize}
\end{ex}
\begin{ex}
	Consider a three-dimensional linear space $ A$ with basis $\{ e_1, e_2, e_3\}$.
	\begin{itemize}
		\item
		The linear space $ A$ equipped with the linear product defined on $ A$ by:
		$e_1\cdot e_1= e_2+e_3 $,
		$e_2\cdot e_2= e_1+e_2-e_3 $,
		$e_3\cdot e_3=-e_1+e_2$
		and for all $i\neq j, e_i\cdot e_j=0$, where $i,j\in \{1, 2, 3\}$,
		is a nearly associative  algebra.
		\item
		The linear space $ A$ equipped with the linear product defined on $ A$ by:
		$e_1\cdot e_1= e_2-e_3 $,
		$e_2\cdot e_2= e_2+e_3 $,
		$e_3\cdot e_3=e_1-e_2+e_3$
		and for all $i\neq j, e_i\cdot e_j=0$, where $i,j\in \{1, 2, 3\}$,
		is a nearly associative  algebra.
		\item
		The linear space $ A$ equipped with the linear product defined on $ A$ by:
		$e_2\cdot e_2= e_1+e_3 $,
		$e_1\cdot e_1= e_1+e_2+e_3 $,
		$e_3\cdot e_3=e_1+e_2$
		and for all $i\neq j, e_i\cdot e_j=0$, where $i,j\in \{1, 2, 3\}$,
		is a nearly associative  algebra.
	\end{itemize}	
\end{ex}

\begin{defi}[\cite{DassoundoSilvestrov:nearlyhomass:Albert:PowerAssRings,
DassoundoSilvestrov:nearlyhomass:ElduqueMyung:MutatAlternAlg,
DassoundoSilvestrov:nearlyhomass:Myung:Lieadmalg,
DassoundoSilvestrov:nearlyhomass:MyungOkuboSantilli:ApplLieadmalgPhys,
DassoundoSilvestrov:nearlyhomass:Myung:LiealgFlexLieadmalg,
DassoundoSilvestrov:nearlyhomass:Santilli:introLieadmalg,
DassoundoSilvestrov:nearlyhomass:Santilli:Lieadmapprhadronicstr}]
An algebra $(A, \cdot)$ is called Lie admissible if $( A, [.,.])$ is a Lie algebra, where $[x,y]=x\cdot y-y\cdot x$ for all $x,y\in  A$.
\end{defi}
For a Lie admissible algebra $( A, \cdot)$, the Lie algebra $\mathcal{G}(A)=( A, [.,.])$ is called an underlying Lie algebra of $( A, \cdot)$.

It is known that associative algebras, left-symmetric algebras and anti-flexible algebras (center-symmetric algebras) are Lie-admissible \cite{DassoundoSilvestrov:nearlyhomass:Bai:LeftsymbialgYangBaxtereq,DassoundoSilvestrov:nearlyhomass:BaiC:doublconstrfrobalg, DassoundoSilvestrov:nearlyhomass:HounkonnouDassoundo:centersymalgbialg}.

\begin{pro} \label{prop_Lie_adm}
Any nearly associative algebra is Lie-admissible.
\end{pro}
\begin{proof}
For $[.,.]:(v,w)\mapsto v\cdot w-w\cdot v$ and $x, y, z$ in a nearly associative algebra $( A, \cdot)$,
\begin{eqnarray*}
&& [x,[y,z]]+[y, [z,x]]+[z, [x,y]]\cr
&&  =[x, y\cdot z-z\cdot y]+[y, z\cdot x-x\cdot z]+[z, x\cdot y-y\cdot x]\cr
&& =x\cdot (y\cdot z)-x\cdot(z\cdot y)-(y\cdot z)\cdot x+(z\cdot y)\cdot x\cr
&& \quad +y\cdot (z\cdot x)-y\cdot (x\cdot z)-(z\cdot x)\cdot y+(x\cdot z)\cdot y\cr
&& \quad +z\cdot(x\cdot y)-z\cdot (y\cdot x)-(x\cdot y)\cdot z+(y\cdot x)\cdot z\cr
&& =\{x\cdot (y\cdot z)-(z\cdot x)\cdot y \}+\{(y\cdot x)\cdot z-x\cdot(z\cdot y)\}\cr
&& \quad +\{y\cdot (z\cdot x)-(x\cdot y)\cdot z\}+\{z\ast(x\cdot y)-(y\cdot z)\cdot x\}\cr
&& \quad +\{(z\cdot y)\cdot x-y\cdot (x\cdot z)\}+\{(x\cdot z)\cdot y-z\cdot (y\cdot x) \}=0.
\end{eqnarray*}
Therefore, $( A, [.,.])$ is a Lie algebra.
\end{proof}

\begin{rmk}
In a nearly associative algebra $( A, \cdot)$, for $x,y\in  A$,
\begin{subequations}
\begin{eqnarray}
		L(x)L(y)&=&R(y)R(x), \\
		L(x)R(y)&=&L({y\cdot x}), \\
		R(x)L(y)&=&R(x\cdot y),
\end{eqnarray}
\end{subequations}
where $L, R:  A\rightarrow \End( A)$ are
the operators of left and right multiplications.
\end{rmk}
\begin{defi}
An 	anti-flexible  algebra is a couple $(A, \cdot)$	where
$ A$ is a linear space, and  $\cdot:  A\times  A\rightarrow A$
is a bilinear product such that for all $x, y, z\in  A$,
	\begin{eqnarray}\label{eq_associative}
	(x\cdot y)\cdot z-(z\cdot y)\cdot x=x\cdot (y\cdot z)-z\cdot (y\cdot x).
	\end{eqnarray}
\end{defi}
Using associator $a(x,y,z)= (x\cdot y)\cdot z-x\cdot (y\cdot z)$, the equality \eqref{eq_associative} is equivalent to
	\begin{eqnarray}\label{eq_associative_ass}
	a(x,y,z)=a(z, y, x).
	\end{eqnarray}
In view of \eqref{eq_associative_ass}, anti-flexible algebras were called center-symmetric algebras in \cite{DassoundoSilvestrov:nearlyhomass:HounkonnouDassoundo:centersymalgbialg}.

\begin{pro}
Any commutative nearly associative  algebra is anti-flexible.
\end{pro}
\begin{proof}
For all $x,y,z\in  A$ in a commutative nearly associative  algebra $( A, \cdot )$,
by using nearly associativity, commutativity and again nearly associativity,
\begin{eqnarray*}
a(x,y,z)=(x\cdot y)\cdot z - x\cdot (y\cdot z)= y\cdot (z\cdot x)-(z\cdot x)\cdot y
=[y,z\cdot x]=[y,x\cdot z]= \cr
y\cdot(x\cdot z) - (x\cdot z) \cdot y = (z\cdot y) \cdot x - z \cdot (y\cdot x) = a(z,y,x)
\end{eqnarray*}
proves \eqref{eq_associative_ass} meaning that $( A, \cdot )$ is anti-flexible.
\end{proof}


\section{Classification of the two-dimensional nearly associative  algebras}\label{section2}
\begin{thm}\label{thm_fond}
Any two-dimensional algebra  $(A, \cdot)$ is nearly associative
if and only if
\begin{eqnarray*}
	e_1\cdot(e_1\cdot e_1)=(e_1\cdot e_1)\cdot e_1,\qquad
	e_1\cdot(e_1\cdot e_2)=(e_2\cdot e_1)\cdot e_1,\\
	e_1\cdot(e_2\cdot e_1)=(e_1\cdot e_1)\cdot e_2,\qquad
	e_2\cdot(e_1\cdot e_1)=(e_1\cdot e_2)\cdot e_1,\\
	e_1\cdot(e_2\cdot e_2)=(e_2\cdot e_1)\cdot e_2,\qquad
	e_2\cdot(e_1\cdot e_2)=(e_2\cdot e_2)\cdot e_1,\\
	e_2\cdot(e_2\cdot e_1)=(e_1\cdot e_2)\cdot e_2,\qquad
	e_2\cdot(e_2\cdot e_2)=(e_2\cdot e_2)\cdot e_2,
\end{eqnarray*}
where $\{e_1, e_2\}$ is a basis of $A.$
\end{thm}
\begin{thm}\label{Thm_Classification_two_dim}
Any two-dimensional nearly associative  algebra is isomorphic to one of the following nearly associative  algebras:
\begin{itemize}
\item 	For all $(\alpha, \beta)\in \mathbb{K}^2\backslash\{ (0,0)\}$, $e_1\cdot e_1=\alpha e_2  , e_1\cdot e_2= \beta e_1= e_2\cdot e_1, e_2\cdot e_2=\beta e_2$.
\item 	For all $(\alpha, \beta)\in \mathbb{K}^2\backslash\{ (0,0)\}$, $e_1\cdot e_1=\alpha e_1+\beta e_2, e_1\cdot e_2= \beta e_1+\alpha e_2= e_2\cdot e_1, \\e_2\cdot e_2=\alpha e_1+\beta e_2$.
\item 	For all $(\alpha, \beta, \gamma)\in \mathbb{K}^3$, such that $\gamma^2+4\alpha\beta\geq0,$ $e_1\cdot e_1=\alpha e_1,
e_2\cdot e_2=\beta e_1+\gamma e_2,\\
e_1\cdot e_2=\frac{1}{2}\left(\gamma+
\sqrt{\gamma^2+4\alpha\beta}\right) e_1= e_2\cdot e_1.
		$
\end{itemize}
\end{thm}
\begin{proof}

Equip the linear space $A$ with the basis $\{e_1, e_2\}$, and for all $i;j\in \{1;2\}$, set
$e_i\cdot e_j=a_{ij}e_1+b_{ij}e_2$, where $a_{ij}\in \mathbb{K}$ and $b_{ij}\in \mathbb{K}$.
In addition, for all $i, j, k\in \{1,2\}$,
$
a_{jk}a_{i1}+b_{jk}a_{i2}=a_{ki}a_{1j}+b_{ki}a_{2j},
a_{jk}b_{i1}+b_{jk}b_{i2}=a_{ki}b_{1j}+b_{ki}b_{2j}.
$
By Theorem~\ref{thm_fond},
\begin{eqnarray*}
&\left\lbrace
\begin{array}{lllllllllllllllll}
	e_1\cdot(e_1\cdot e_1)=(e_1\cdot e_1)\cdot e_1 \\
	e_1\cdot(e_1\cdot e_2)=(e_2\cdot e_1)\cdot e_1 \\
	e_1\cdot(e_2\cdot e_1)=(e_1\cdot e_1)\cdot e_2 \\
	e_2\cdot(e_1\cdot e_1)=(e_1\cdot e_2)\cdot e_1 \\
	e_1\cdot(e_2\cdot e_2)=(e_2\cdot e_1)\cdot e_2 \\
	e_2\cdot(e_1\cdot e_2)=(e_2\cdot e_2)\cdot e_1 \\
	e_2\cdot(e_2\cdot e_1)=(e_1\cdot e_2)\cdot e_2 \\
	e_2\cdot(e_2\cdot e_2)=(e_2\cdot e_2)\cdot e_2
\end{array}
\right. \\*[0,3cm]
& \Longleftrightarrow
\left\lbrace
\begin{array}{lllllllllllllllll}
	a_{11}a_{11}+b_{11}a_{12}=a_{11}a_{11}+b_{11}a_{21},\\
	a_{11}b_{11}+b_{11}b_{12}=a_{11}b_{11}+b_{11}b_{21},\\
	a_{12}a_{11}+b_{12}a_{12}=a_{21}a_{11}+b_{21}a_{21},\\
	a_{12}b_{11}+b_{12}b_{12}=a_{21}b_{11}+b_{21}b_{21},\\
	a_{21}a_{11}+b_{21}a_{12}=a_{11}a_{12}+b_{11}a_{22},\\
	a_{21}b_{11}+b_{21}b_{12}=a_{11}b_{12}+b_{11}b_{22},\\
	a_{11}a_{21}+b_{11}a_{22}=a_{12}a_{11}+b_{12}a_{21},\\
	a_{11}b_{21}+b_{11}b_{22}=a_{12}b_{11}+b_{12}b_{21},\\
	a_{22}a_{11}+b_{22}a_{12}=a_{21}a_{12}+b_{21}a_{22},\\
	a_{22}b_{11}+b_{22}b_{12}=a_{21}b_{12}+b_{21}b_{22},\\
	a_{12}a_{21}+b_{12}a_{22}=a_{22}a_{11}+b_{22}a_{21},\\
	a_{12}b_{21}+b_{12}b_{22}=a_{22}b_{11}+b_{22}b_{21},\\
	a_{21}a_{21}+b_{21}a_{22}=a_{12}a_{12}+b_{12}a_{22},\\
	a_{21}b_{21}+b_{21}b_{22}=a_{12}b_{12}+b_{12}b_{22},\\
	a_{22}a_{21}+b_{22}a_{22}=a_{22}a_{12}+b_{22}a_{22},\\
	a_{22}b_{21}+b_{22}b_{22}=a_{22}b_{12}+b_{22}b_{22}
\end{array}
\right. \\*[0,3cm]
&\Longleftrightarrow
\left\lbrace
\begin{array}{lllllllllllllllll}
	e(b-c)=0, e(f-g)=0,\\
	h(b-c)=0,
	d(b-c)=0,\\
	d(f-g)=0, a(f-g)=0\\
	(b-c)(b+c)=0,\\
	(f-g)(f+g)=0\\
	e(b-c)+f(g-a)=0\\
	d(a-g)+b(h-c)=0\\
	a(b-c)=0, h(f-g)=0\\
	(bf-cg)=0, bg=de=fc
\end{array}
\right. \\
& \Longleftrightarrow
\left\lbrace
\begin{array}{cccccccccccccc}
	\left\lbrace
	\begin{array}{llllllllllll}
		a= r_1,b= r_2,c= r_2,d= r_1,\\
		e= r_2,f= r_1,g= r_1,h= r_2\\
	\end{array}
	\right.
	\mbox{or }
	\left\lbrace
	\begin{array}{llllllllllll}
		a= r_1,b= r_2,c= r_2,d= r_2,\\
		e= r_1,f= r_1,g= r_1,h= r_2\\
	\end{array}
	\right.
	\\
	\mbox{or }
	\\
	\left\lbrace
	\begin{array}{lllllllllll}
		a=  r_5,b=0,c=0,d=   0,\\
		e=  r_6,f=  r_5,g=   r_5,h=  r_8\\
	\end{array}
	\right.
	\mbox{or }
	\left\lbrace
	\begin{array}{lllllllllll}
		a=  r_5,b=0,c=0,d=   r_6,\\
		e=  0,f=  r_5,g=   r_5,h=  r_8\\
	\end{array}
	\right.
	\\
	\mbox{or }
	\\
	\left\lbrace
	\begin{array}{llllllllll}
		a=   r_9,b=\frac{|r_{12}|+  r_{12}}{2}, \\
		c=\frac{|r_{12}|+  r_{12}}{2},
		d=   0,\\ e= r_{11},f=0,g=0,h=   r_{12}\\
	\end{array}
	\right.
	\mbox{or }
	\left\lbrace
	\begin{array}{llllllllll}
		a=   r_9,b=\frac{\sqrt{4   r_{10}\,    r_9+{{   r_{12}}^{2}}}+  r_{12}}{2},\\
		c=\frac{\sqrt{4  r_{10}\,   r_9 +{{  r_{12}}^{2}}}+  r_{12}}{2},
		d=   r_{10},\\e=0,f=0,g=0,h=   r_{12}
	\end{array}
	\right.
	\\
	\mbox{or }
	\\
	\left\lbrace
	\begin{array}{lllllllllllllll}
		a= r_ { 13},b=\frac{  r_{ 16}-\sqrt{{{  r_{ 16}}^{2}}}}{2},\\
		c=\frac{   r_{16}-\sqrt{{{   r_{16}}^{2}}}}{2},
		d=  0,\\e= r_ { 14},f=0\\,g=0,h=  r_{ 16}\\
	\end{array}
	\right.
	\mbox{or }
	\left\lbrace
	\begin{array}{lllllllllllllll}
		a= r_ { 13},b=\frac{  r_{ 16}-\sqrt{{{  r_{ 16}}^{2}}+4   r_{13}\,   r_{14}}}{2},\\
		c=\frac{   r_{16}-\sqrt{{{   r_{16}}^{2}}+4    r_{13}\,   r_{ 14}}}{2},
		d=  r_{14},\\e= 0,f=0,g=0,h=  r_{ 16}\\
	\end{array}
	\right.
	\\
	\mbox{or }
	\\
	\left\lbrace
	\begin{array}{llllllllllllllll}
		a=  r_{ 17},b=0,
		c=0,
		d= 0,\\e=  r_{18},f=0,g=0,h=0\\
	\end{array}
	\right.
	\mbox{or }
	\left\lbrace
	\begin{array}{llllllllllllllll}
		a=  r_{ 17},b=\sqrt{  r_{ 17}\,   r_{18}},
		c=\sqrt{  r_{ 17}\,   r_{ 18}},\\
		d=   r_{18},e= 0,f=0,g=0,h=0\\
	\end{array}
	\right.
	\\
	\mbox{or }
	\\
	\left\lbrace
	\begin{array}{lllllllllllllll}
		a= r_{20},b=0,
		c=0,d=  0,\\
		e=   r_{21},f=0,g=0,h=0\\
	\end{array}
	\right.
	\mbox{or }
	\left\lbrace
	\begin{array}{lllllllllllllll}
		a= r_{20},b=-\sqrt{  r_{20}\,    r_{21}}, \\
		c=-\sqrt{  r_{20}\,  r_{ 21}},d=  r_{ 21},\\
		e= 0,f=0,g=0,h=0\\
	\end{array}
	\right.
	\\
	\mbox{or }
	\\
	\left\lbrace
	\begin{array}{lllllllllllll}
		a=0,b=0,c=0,d= 0,\\e=  r_{24},f=0,g=0,h=   r_{25}\\
	\end{array}
	\right.
	\mbox{or }
	\left\lbrace
	\begin{array}{lllllllllllll}
		a=0,b=0,c=0,d=  r_{23},\\e= 0,f=0,g=0,h=   r_{25}\\
	\end{array}
	\right.
	\\
	\mbox{or }
	\\
	\left\lbrace
	\begin{array}{lllllllllllllll}
		a=0,b=  r_{26},c=  r_{ 26},d=  0,\\
		e=  r_{ 28},f=0,g=0,h=   r_{26}\\
	\end{array}
	\right.
	\mbox{or }
	\left\lbrace
	\begin{array}{lllllllllllllll}
		a=0,b=  r_{26},c=  r_{ 26},d=  r_{ 27},\\
		e= 0,f=0,g=0,h=   r_{26}\\
	\end{array}
	\right.
	\\
	\mbox{or }
	\\
	\left\lbrace
	\begin{array}{llllllllllllllll}
		a=  r_{ 29},b=0,c=0,d=0, \\
e=0,f=0,g=0,h=0
	\end{array}
	\right.
\end{array}
\right.
\end{eqnarray*}
with
$
a_{11}=a, a_{12}=b, a_{21}=c, a_{22}=d,
b_{11}=e, b_{12}=f, b_{21}=g, b_{22}=h.
$

Therefore, the non-isomorphic algebras generated by these constants structures are:
\begin{itemize}
	\item
	For all $(\alpha, \beta)\in \mathbb{K}^2\backslash\{ (0,0)\}$,$e_1\cdot e_1=\alpha e_2  , e_1\cdot e_2= \beta e_1= e_2\cdot e_1, e_2\cdot e_2=\beta e_2$.
	\item
	For all $(\alpha, \beta)\in \mathbb{K}^2\backslash\{ (0,0)\}$,$e_1\cdot e_1=\alpha e_1+\beta e_2, e_1\cdot e_2= \beta e_1+\alpha e_2= e_2\cdot e_1, \\e_2\cdot e_2=\alpha e_1+\beta e_2$.
	\item
	For all $(\alpha, \beta, \gamma)\in \mathbb{K}^3$, such that $\gamma^2+4\alpha\beta\geq0,$
	$e_1\cdot e_1=\alpha e_1,
	e_2\cdot e_2=\beta e_1+\gamma e_2,\\
	e_1\cdot e_2=\frac{1}{2}\left(\gamma+\sqrt{\gamma^2+4\alpha\beta}\right) e_1= e_2\cdot e_1.
	$
\end{itemize}
\end{proof}

\section{Bimodules and matched pairs nearly associative  algebras}\label{section3}
\begin{defi}
Let $( A, \cdot )$ be a nearly associative  algebra. Consider the linear maps
$l;r: A\rightarrow \End(V)$, where $V$ is a linear space.
A triple $( l, r, V)$ is a bimodule of
$( A, \cdot)$ if for all $x, y\in  A$, the following relations
\begin{subequations}
\begin{eqnarray}\label{eq_bimodule_1_1}
l(x)l(y)&=&r(y)r(x),
\\
\label{eq_bimodule_1_2}
l(x)r(y)&=&l({y\cdot x}),
\\
\label{eq_bimodule_1_3}
r(x)l(y)&=&r({x\cdot y})
\end{eqnarray}
\end{subequations}
are satisfied.
\end{defi}
\begin{ex}
Let $( A, \cdot )$ be a nearly associative  algebra. The triple
$(L, R,  A)$ is a bimodule of $( A, \cdot)$, where for any $x,y\in A$,
$L(x)y=x\cdot y=R(y)x$.
\end{ex}
\begin{pro}
	Let $(l, r, V)$ be a bimodule of a nearly associative  algebra $( A, \cdot)$, where
	$l;r: A\rightarrow \End(V)$ are two linear maps and  $V$  a linear space.
	There is a nearly associative  algebra defined on $ A\oplus V$ by,
	for any $x,y\in  A$ and any $u,v\in V,$
	\begin{eqnarray}\label{eq_struc_bimodule_1}
	(x+u)\ast (y+v)=x\cdot y+ l(x)v+r(y)u.
	\end{eqnarray}
\end{pro}
\begin{proof}

Consider the bimodule $(l,r, V)$ of the nearly associative  algebra $( A, \cdot )$. For all $x,y,z\in  A$ and $u, v,w \in V$ we have:
\begin{subequations}
\begin{eqnarray}\label{eq_proof_rep1}
(x+u)\ast((y+v)\ast(z+w))
= x\cdot (y\cdot z)+l(x)l(y)w+l(x)r(z)v+r({y\cdot z})u
\\
\label{eq_proof_rep2}
((z+w)\ast (x+u))\ast (y+v)
=(z\cdot x)\cdot y+l({z\cdot x})v+r(y)l(z)u+r(y)r(x)w
\end{eqnarray}
\end{subequations}
Using  \eqref{eq_bimodule_1_1}~-~\eqref{eq_bimodule_1_3}
in  \eqref{eq_proof_rep1} and  \eqref{eq_proof_rep2}
we easily deduce that $(A\oplus V, \ast)$
is a nearly associative  algebra.
\end{proof}
\begin{cor}	
Let $(l, r, V)$ be a bimodule of a nearly associative  algebra
$(A, \cdot)$, where $l,r:\A\rightarrow \End(V)$
are two linear maps and $V$  a linear space.	
Then there is a Lie algebra product on $A\oplus V$ given
by
\begin{eqnarray}\label{eq_struc_lie1}
[x+u, y+v]=[x,y]_{_\cdot }+(l(x)-r(x))v-(l(y)-r(y))u
\end{eqnarray}
for all $x, y\in A$ and  for any $u,v\in V$.
\end{cor}
\begin{proof}
It is simple to remark that the commutator of the
product defined in  \eqref{eq_struc_bimodule_1}
is the product defined in  \eqref{eq_struc_lie1}.
By taking into account Proposition~\ref{prop_Lie_adm},
the Jacobi identity of the product given in
 \eqref{eq_struc_lie1} is satisfied.
\end{proof}
\begin{defi}
Let $(\mathcal{G}, [.,.]_{_\mathcal{G}} )$ be a Lie algebra.
A representation of $(\mathcal{G}, [.,.]_{_\mathcal{G}} )$ over the linear
space $V$ is a linear map
$\rho: \mathcal{G} \rightarrow \End(V)$ satisfying
\begin{eqnarray}\label{eq_representation_Lie}
\rho([x, y]_{_\mathcal{G}}) = \rho(x) \circ \rho(y) - \rho(y)\circ \rho(x)
\end{eqnarray}
for all $x, y \in \mathcal{G}$.
\end{defi}
\begin{pro}\label{proposition_representation_Lie}
Let $( A, \cdot)$ be a nearly associative
algebra and let $V$ be a finite-dimen\-sional linear
space over the field $\mathbb{K}$ such that $(l, r, V )$
is a bimodule of $( A, \cdot)$, where
$l, r :  A \rightarrow \End(V )$ are two
linear maps. Then, the linear map
$$l - r : A \rightarrow \End(V ), \quad  x \mapsto l(x) - r(x)$$
is a representation of the underlying Lie algebra
$\mathcal{G}( A)$ underlying $( A, \cdot)$.
\end{pro}
\begin{proof}

Let $(l,r, V)$ be a bimodule of the nearly associative
algebra $( A, \cdot)$.
For $x,y\in  A$,
\begin{multline*}
(l(x)-r(x))(l(y)-r(y))-(l(y)-r(y))(l(x)-r(x))= \cr
l(x)l(y)-l(x)r(y)-r(x)l(y)+r(x)r(y) -l(y)l(x)+l(y)r(x)+r(y)l(x)-r(y)r(x)\cr
= -l(x)r(y)-r(x)l(y)+l(y)r(x)+r(y)l(x)
=-l({y\cdot x})-r({x\cdot y})+l({x\cdot x})+r({y\cdot x})\cr
= (l-r)(x\cdot y-y\cdot x)=(l-r)([x,y]).
\end{multline*}
Therefore,  \eqref{eq_representation_Lie}
is satisfied for $l - r=\rho$.
\end{proof}
\begin{defi}
Let $( A, \cdot)$ be a nearly associative
algebra and   $(l, r, V)$ its associated a
bimodule, where $V$ is a finite-dimensional linear
space. The dual maps $l^{*}, r^{*}$
of linear maps $l, r,$ respectively  are defined as
$\displaystyle l^{*}, r^{*}:  A \rightarrow \End(V^{*})$  such that
for any $x \in  A, u^{*} \in V^{*}, v \in V,$
\begin{subequations}	
\begin{eqnarray}\label{dual1}
\left< l^{*}(x)u^{*}, v \right> = \left< u^{*}, l(x) v\right>,
\\
\label{dual2}
\left<  r^{*}(x)u^{*}, v \right> = \left<u^{*},   r(x) v \right>.
\end{eqnarray}
\end{subequations}
\end{defi}
\begin{pro}\label{Prop_Dual_bimodule}
Let $( A, \cdot)$ be a nearly associative
algebra and $(l,r, V)$ be its bimodule.
The following relations are equivalent:
\begin{enumerate}[label=\upshape{(\roman*)},left=7pt]
\item $(r^*,l^*, V^*)$ is a bimodule of $( A, \cdot)$,
\item    $l(x)r(y)=r(y) l(x)$, for all $x,y \in  A$,
\item $(l^*,r^*, V^*)$ is a bimodule of $( A, \cdot)$.
\end{enumerate}
\end{pro}
\begin{proof}

Let $( A, \cdot )$  be a nearly associative
algebra and $(l,r, V)$ be its associated
bimodule i.e. the linear maps
$l,r: A \rightarrow \End(V)$ satisfying
 \eqref{eq_bimodule_1_1}~-~\eqref{eq_bimodule_1_3}
and $V$ is a finite-dimensional linear space.
\begin{itemize}
\item
Suppose that $(r^*,l^*, V^*)$ is a bimodule of $( A, \cdot)$, i.e. with correspondences $l \rightarrow r^*$ and $r\rightarrow l^*$,
 \eqref{eq_bimodule_1_1}~-~\eqref{eq_bimodule_1_3} are satisfied.
For any $x,y\in  A$, $v\in V$, $u^*\in V^*$:
\begin{multline*}
\langle l(x)r(y)v, u^*\rangle=\langle v, r^*(y)l^*(x)u^*\rangle
=\langle v, r^*({y\cdot x})u^*\rangle \\
=\langle r({y\cdot x})v, u^*\rangle
=\langle r(y)l(x)v, u^*\rangle.
\end{multline*}
Therefore, the relation $l(x)r(y)=r(y)l(x)$ is satisfied.
\item
Suppose $l(x)r(y)=r(y)l(x)$ for any $x,y \in  A$.
For any $x,y\in  A$, $v\in V$, $u^*\in V^*$:
\begin{multline*}
\left<l^*(x)l^*(y)u^*, v\right>=\left<u^*,l(y)l(x)v \right>
=\left<u^*,r(x)r(y)v \right>=\left<r^*(y)r^*(x)u^*, v\right>,
\end{multline*}
yields $l^*(x)l^*(y)=r^*(y)r^*(x);$
\begin{multline*}
\left<l^*(x)r^*(y)u^*,v \right>=\left<u^*,r(y)l(x)v \right>
=\left<u^*,l(x)r(y)v \right>= \\
\left<u^*,l({y\cdot x})v \right>=\left<l^*({y\cdot x})u^*, v\right>,
\end{multline*}
yields $l^*(x)r^*(y)=l^*({y\cdot x})$;
\begin{multline*}
\left<r^*(y)l^*(x)u^*,v \right>=\left<u^*, l(x)r(y)v\right>=\left<u^*, r(y)l(x)v\right>= \\
\left<u^*, r({y\cdot x})v\right>=\left<r^*({y\cdot x})u^*, v\right>,
\end{multline*}
yields $r^*(y)l^*(x)=r^*({y\cdot x}).$
Thus, with correspondences $r^* \rightarrow l $ and $ l^*\rightarrow  r $,
\eqref{eq_bimodule_1_1}~-~\eqref{eq_bimodule_1_3} are satisfied.
\end{itemize}
Similarly, one obtains the equivalence between
$l(x) r(y)=r(y) l(x)$, for any $x,y \in  A$,
and $(l^*,r^*, V^*)$ being a bimodule of $( A, \cdot)$.
\end{proof}
\begin{rmk}\label{rmk_useful}
It is clear that $(L_{\cdot}^*, R_{\cdot}^*,  A^*)$ and
$(R_{\cdot}^*, L_{\cdot}^*,  A^*)$
are bimodules of the nearly associative algebra $( A, \cdot)$
if and only if $L$ and $R$ commute.
\end{rmk}
\begin{thm}\label{theo_matched_pair}
Let $( A, \cdot)$ and $(B, \circ)$ be two nearly associative  algebras.
Suppose that $(l_{ A}, r_{ A}, B)$ and $(l_{B}, r_{B},  A)$
are bimodules of $( A, \cdot)$
and $(B, \circ)$, respectively, where
$l_{ A}, r_{ A}:  A\rightarrow \End(B) $,
$l_{B}, r_{B}:B\rightarrow\End( A)$
are four linear maps satisfying for all $x, y\in  A$, $a,b\in B$ the following relations
\begin{subequations}
\begin{eqnarray}\label{eq_matched_pair_1_1}
r_{B}(l_{ A}(x)a)y+y\cdot (r_{B}(a)x)
-(l_{B}(a)y)\cdot x-l_{B}(r_{ A}(y)a)x=0,
\\
\label{eq_matched_pair_1_2}
r_{B}(a)(x\cdot y)-y\cdot (l_{B}(a)x)-r_{B}(r_{ A}(x)a)y=0,
\\
\label{eq_matched_pair_1_3}
l_{B}(a)(x\cdot y)-(r_{B}(a)y)\cdot x-l_{B}(l_{ A}(y)a)x=0,
\\
\label{eq_matched_pair_1_4}
r_{ A}(l_{B}(a)x)b+b\circ (r_{ A}(x)a)
-(l_{ A}(x)b)\circ a-l_{ A}(r_{B}(b)x)a=0,
\\
\label{eq_matched_pair_1_5}
	r_{ A}(x)(a\circ b)-b\circ(l_{ A}(x)a)-r_{ A}(r_{B}(a)x)b=0,
\\
\label{eq_matched_pair_1_6}
	l_{ A}(x)(a\circ b)-(r_{ A}(x)b)\circ a-l_{ A}(l_{B}(b)x)a=0.
\end{eqnarray}
\end{subequations}
	Then, $( A\oplus B, \ast)$ is a nearly associative  algebra,
	where
	\begin{eqnarray}\label{eq_bilinear_1}
	(x+a)\ast(y+b)=
	(x\cdot y+l_{B}(a)y+r_{B}(b)x)+
	(a\circ b+l_{ A}(x)b+r_{ A}(y)a).
	\end{eqnarray}
for all $x,y\in  A, a, b\in B$.
\end{thm}
\begin{proof}
For any $x,y,z\in  A$, and for any $a,b,c\in B$, we have
\begin{multline*}
(x+a)\ast ((y+b)\ast(z+c))= x\cdot (y\cdot z)+
\{x\cdot (l_{ B}(b)z)+r_{ B}(r_{ A}(z)b)x \} +l_{ B}(a)(y\cdot z)
\cr
+
\{x\cdot (r_{ B}(c)y)+r_{ B}(l_{ A}(y)c)x \}+l_{ B}(a)(l_{ B}(b)z)+r_{ B}(b\circ c)x\cr
+l_{ B}(a)(r_{ B}(c)y)
+a\circ (b\circ c)+\{a\circ (l_{ A}(y)c)+r_{ B}(r_{ B}(c)y)a\}\cr
+\{a\circ (r_{A}(z)b)+r_{ B}(l_{ B}(b)z)a\}+l_{ A}(x)(l_{ A}(y)c)
+ l_{ A}(x)(b\circ c)+l_{ A}(x)(r_{ A}(z)b)+r_{ A}(y\cdot z)a;
\\*[0,3cm]
((z+c)\ast (x+a)  )\ast (y+b)=
(z\cdot x)\cdot y+\{(l_{ B}(c)x)\cdot y+l_{ B}(r_{ A}(x)c)y   \}+r_{ B}(x)(z\cdot x)\cr
+ \{ (r_{ A}(a)z)\cdot y+l_{ B}(l_{ A}(z)a)y \}+l_{ B}(c\circ a)y+r_{ B}(b)(l_{ B}(c)x)\cr
+r_{ B}(b)(r_{ B}(a)z) (c\circ a)\circ y +\{(l_{ A}(z)a)\circ b +l_{ A}(r_{ B}(a)z)b\}\cr
+\{(r_{ A}(x)c)\circ b+l_{ A}(l_{ B}(c)x)b \}+r_{ A}(y)(r_{ A}(x)c)
+r_{ A}(y)(c\circ a)+r_{ A}(y)(l_{ A}(z)a)+l_{ A}(z\cdot x)b.
\end{multline*}

Using \eqref{eq_matched_pair_1_1}~-~\eqref{eq_matched_pair_1_6} and
that $(l_{ A}, r_{ A},  B)$ and $(l_{ B}, r_{ B},  A)$
are bimodules of $( A,\cdot )$ and $( B, \circ )$, respectively,
we derive that
 $( A\oplus B, \ast)$ is a nearly associative  algebra.
\end{proof}
\begin{defi}[\cite{DassoundoSilvestrov:nearlyhomass:Majid:matchedpeirsLiegrYangBa}]\label{Dfn_Matched_pair_Lie}
	Let $(\mathcal{G}, [.,.]_{_\mathcal{G}})$ and $({\mathcal{H}, [.,.]_{_\mathcal{H}}})$ be two Lie algebras such that
	$\rho:\mathcal{G}\rightarrow \End(\mathcal{H})$ and $\mu:\mathcal{H}\rightarrow \End(\mathcal{G})$ are representations of
	$\mathcal{G}$ and $\mathcal{H}$, respectively.
 A matched pair of Lie algebras  $\mathcal{G}$ and $\mathcal{H}$ is 	$(\mathcal{G}, \mathcal{H}, \rho, \mu)$
 such that $\rho$ and $\mu$ are satisfying the following relations, for all $x, y \in \mathcal{G}$ and $a, b \in  \mathcal{H}$,
\begin{subequations}
	\begin{eqnarray}\label{eq_matehd_pair_Lie_1}
	\rho(x)[a, b]_{_\mathcal{G}} - [\rho(x)a, b]_{_\mathcal{H}} - [a, \rho(x)b]_{_\mathcal{H}} + \rho(\mu(a)x)b - \rho(\mu(b)x)a = 0,
	\\
\label{eq_matehd_pair_Lie_2}
	\mu(a)[x, y]_{_\mathcal{G}} - [\mu(a)x, y]_{_\mathcal{G}} - [x, \mu(a)y]_{_\mathcal{G}} + \mu(\rho(x)a)y - \mu(\rho(y)a)x = 0.
	\end{eqnarray}
\end{subequations}
\end{defi}
\begin{cor}
Let $( A,  B, l_{ A}, r_{ A}, l_{ B}, r_{ B})$ be a matched pair
of the nearly associative  algebras $( A, \cdot)$ and $( B, \circ )$.
Then $(\mathcal{G}( A), \mathcal{G}( B), l_{ A}-r_{ A}, l_{ B}-r_{ B})$ is a matched pair
of Lie algebras $\mathcal{G}( A)$ and $\mathcal{G}( B)$.
\end{cor}
\begin{proof}

Let $( A,  B, l_{ A}, r_{ A}, l_{ B}, r_{ B})$ be a matched pair of
the nearly associative  algebras $( A, \cdot)$ and $( B, \circ )$.
In view of Proposition~\ref{proposition_representation_Lie},
the linear maps $l_{ A}-r_{ A}:  A\longrightarrow \End( B)$ and
$l_{ B}-r_{ B}:  B\longrightarrow\End( A)$ are representations of
the underlying Lie algebras $\mathcal{G}( A)$ and $\mathcal{G}( B)$, respectively. Therefore,
by direct calculation we have   \eqref{eq_matehd_pair_Lie_1} is equivalent to
 \eqref{eq_matched_pair_1_1}~-~\eqref{eq_matched_pair_1_3} and similarly,
 \eqref{eq_matehd_pair_Lie_2} is equivalent to
 \eqref{eq_matched_pair_1_4}~-~\eqref{eq_matched_pair_1_6}.
\end{proof}
\begin{pro}
Let $(A,\cdot)$ be a nearly associative algebra. Suppose that there is
a nearly associative algebra structure $\circ$ on its the dual
space $A^*$. If in addition, the linear maps $L$ and  $R$ commute then
 $(A, A^*, R_{\cdot}^*, L_{\cdot}^*, R_{\circ}^*, L_{\circ}^*)$ is
 a matched pair of the nearly associative algebras
 $(A, \cdot)$ and $(A^*, \circ)$ if and only if the following relations
 are satisfied for any $x,y\in A$ and $a\in A^*$
 \begin{subequations}
 \begin{eqnarray}\label{eq:matchedpair1}
 L_{\circ}^*(R_{\cdot}^*(x)a)y-y\cdot (L_{\circ}^*(a)x)-
 (R_{\circ}^*(a)y)\cdot x-R_{\circ}^*(L_{\cdot}^*(y)a)x=0,
 \\
 \label{eq:matchedpair2}
 L_{\circ}^*(a)(x\cdot y)-y\cdot (R_{\circ}^*(a)x)-L_{\circ}^*(L_{\cdot}^*(x)a)y=0,
 \\
 \label{eq:matchedpair3}
 R_{\circ}^*(a)(x\cdot y)-(L_{\circ}^*(a)y)\cdot x-R_{\circ}^*(R_{\cdot}^*(y)a)x=0.
\end{eqnarray}
 \end{subequations}
\end{pro}
\begin{proof}
Since $L$ and $R$ commute, according to Remark~\ref{rmk_useful} and
Proposition~\ref{Prop_Dual_bimodule}, both
$(R_{\cdot}^*, L_{\cdot}^*, A^*)$ and $(L_{\cdot}^*, R_{\cdot}^*, A^*)$
are bimodules of $(A, \cdot)$. Setting
$l_{A}=R_{\cdot}^*$, $r_{A}=L_{\cdot}^*, l_{B}=R_{\circ}^*$ and
$r_{B}=L_{\circ}^*$ in Theorem~\ref{theo_matched_pair}
the equivalences among
Eq~\eqref{eq_matched_pair_1_1} and  \eqref{eq:matchedpair1},
 \eqref{eq_matched_pair_1_2} and  \eqref{eq:matchedpair2}, and finally
 \eqref{eq_matched_pair_1_3} and  \eqref{eq:matchedpair3}
are straightforward. Besides, for any $x,y\in A$ and any $a,b\in A^*$, we have
\begin{eqnarray*}
&&\langle
L_{\circ}^*(R_{\cdot}^*(x)a)y, b
\rangle=
\langle
y, L_{\circ}(R_{\cdot}^*(x)a)b
\rangle=
\langle
y,(R_{\cdot}^*(x)a)\circ b
\rangle,\cr
&&\langle
y\cdot (L_{\circ}^*(a)x), b
\rangle=
\langle
R_{\cdot}(L_{\circ}^*(a)x)y,b
\rangle=
\langle
y, R_{\cdot}^*(L_{\circ}^*(a)x)b
\rangle,\cr
&&\langle
(R_{\circ}^*(a)y)\cdot x, b
\rangle=
\langle
R_{\circ}^*(a)y,
R_{\cdot}^*(x)b
\rangle=
\langle
y, (R_{\cdot}^*(x)b)\circ a
\rangle, \cr
&&\langle
R_{\circ}^*(L_{\cdot}^*(y)a)x, b
\rangle=
\langle
L_{\circ}^*(b)x, L_{\cdot}^*(y)a
\rangle=
\langle
y\cdot (L_{\circ}^*(b)x), a
\rangle=
\langle
y, R_{\cdot}^*(L_{\circ}^*(b)x)a
\rangle, \cr
&&\langle
L_{\circ}^*(a)(x\cdot y), b
\rangle=
\langle
R_{\cdot}(y)x, a\circ b
\rangle=
\langle
x, R_{\cdot}^*(y)(a\circ b)
\rangle, \cr
&&\langle
y\cdot (R_{\circ}^*(a)x), b
\rangle=
\langle
R_{\circ}^*(a)x,
L_{\cdot}^*(y)b
\rangle=
\langle
x, (L_{\cdot}^*(y)b)\circ a
\rangle, \cr
&&\langle
L_{\circ}^*(L_{\cdot}^*(x)a)y, b
\rangle=
\langle
R_{\circ}^*(b)y, L_{\circ}^*(x)a
\rangle=
\langle
x\cdot (R_{\circ}^*(b)y), a
\rangle=
\langle
x, R_{\cdot}^*(R_{\circ}^*(b)y)a
\rangle, \cr
&&\langle
R_{\circ}^*(a)(x\cdot y), b
\rangle=
\langle
L_{\cdot}(x)y, b\circ a
\rangle=
\langle
y,L_{\cdot}(x)^*(b\circ a)
\rangle, \cr
&&\langle
(L_{\circ}^*(a)y)\cdot x, b
\rangle=
\langle
L_{\circ}^*(a)y, R_{\cdot}^*(b)
\rangle=
\langle
y, a\circ (R_{\cdot}^*(x)b)
\rangle, \cr
&&\langle
R_{\circ}^*(R_{\cdot}^*(y)a)x, b
\rangle=
\langle
L_{\circ}^*(b)x, R_{\cdot}^*(y)a
\rangle=
\langle
(L_{\circ}^*(b)x)\cdot y, a
\rangle=
\langle
y, L_{\cdot}^*(L_{\circ}^*(b)x)a
\rangle
\end{eqnarray*}
Then, Eq~\eqref{eq_matched_pair_1_1} holds if and only if
 \eqref{eq_matched_pair_1_4} holds,
Eq~\eqref{eq_matched_pair_1_2} holds if and only if
 \eqref{eq_matched_pair_1_5} holds, and finally
Eq~\eqref{eq_matched_pair_1_3} holds if and only if
 \eqref{eq_matched_pair_1_6} holds.
\end{proof}
\section{Manin triple and  bialgebra of nearly associative algebras}\label{section4}

\begin{defi}
 A bilinear form $\mathfrak B$
 on a nearly associative  algebra $(A,\cdot)$ is called
 {left-invariant} if $for all x,y,z\in A$,
\begin{equation}
\mathfrak{B}(x\cdot y,z)=\mathfrak{B}(x, y\cdot z).
\end{equation}
\end{defi}

\begin{pro}
Let $(A,\cdot)$ be a nearly associative  algebra. If there is a
nondegenerate symmetric invariant bilinear form $\mathfrak B$
defined on $A$,
then as bimodules of the nearly associative algebra $(A,\cdot)$, $(L,R,
A)$ and $(R^*,L^*, A^*)$ are equivalent. Conversely, if $(L,R, A)$ and
$(R^*,L^*, A^*)$ are equivalent
bimodules of a nearly associative algebra $(A,\cdot)$,
then there exists a nondegenerate invariant bilinear
form $\mathfrak B$ on $A$.
\end{pro}

\begin{defi} A  Manin triple of
nearly associative algebras is a
triple of nearly associative algebras $(A,A_{1},A_{2})$ together with a
nondegenerate symmetric invariant bilinear form $\mathfrak{B}$ on $A$ such that the following conditions are satisfied.
\begin{enumerate}[label=\upshape{(\roman*)},left=7pt]
 \item $A_{1}$ and $A_{2}$
nearly associative subalgebras of $A$;
 \item as linear spaces,
  $A=A_{1}\oplus A_{2}$;
\item  $A_{1}$ and $A_{2}$ are isotropic with
respect to $\mathfrak{B}$, i.e.
for any $x_1,y_1\in A_1$ and any $x_2, y_2\in A_2$,
$\mathfrak{B}(x_1,y_1)=0=\mathfrak{B}(x_2,y_2)=0.$
\end{enumerate}

\end{defi}
\begin{defi} Let  $(A,\cdot)$ be a nearly associative algebra. Suppose that $\circ$  is a
nearly associative algebra structure on the dual space $A^*$ of $A$ and
there is a nearly associative  algebra structure on the direct sum $A\oplus A^*$ of
the underlying linear spaces of $A$ and $A^*$ such that
$(A,\cdot)$ and $(A^*,\circ)$ are subalgebras and the natural
symmetric bilinear form on $A\oplus A^*$ given by $\forall x,y\in A; \forall a^*,b^*\in A^*,$
\begin{equation}\label{eq:sbl}\mathfrak{B}_d(x+a^*,y+b^*):=
\langle  a^*,y\rangle +\langle
x,b^*\rangle ,\;
\end{equation} is
left-invariant, then $(A\oplus A^*,A,A^*)$ is called a
standard Manin triple of nearly associative algebras associated to
$\mathfrak{B}_d$.
\end{defi}
 Obviously, a standard Manin triple of nearly associative algebras
 is a Manin triple of nearly associative algebras.
 By symmetric role of $A$ and $A^*$, we have

\begin{pro}
Every Manin triple of nearly associative algebras is isomorphic to a
standard one.
\end{pro}

\begin{pro} \label{Propo} Let $(A,\cdot)$ be a nearly associative algebra.
Suppose that there is a nearly associative algebra structure
$\circ$ on the dual space $A^*$. There exists a
nearly associative algebra structure on the linear space $A\oplus A^*$
such that  $(A\oplus A^*, A,A^*)$ is a standard  Manin triple of
nearly associative  algebras associated to $\mathfrak{B}_d$ defined by
 \eqref{eq:sbl} if and only if $(A, A^*, R_{\cdot}^*,
L_{\cdot}^*, R_{\circ}^*, L_{\circ}^* )$ is a matched pair of
nearly associative algebras.
\end{pro}

\begin{thm}
Let $(A,\cdot)$ be a nearly associative algebra such that the left and right
multiplication operators commute. Suppose that there is
a nearly associative algebra structure $\circ$ on its the dual
space $A^*$ given by $\Delta^*:A^*\otimes A^*\rightarrow A^*$.
 Then,  $(A, A^*, R_{\cdot}^*, L_{\cdot}^*, R_{\circ}^*, L_{\circ}^*)$ is
 a matched pair of the nearly associative algebras
 $(A, \cdot)$ and $(A^*, \circ)$ if and only if
 $\Delta:A\rightarrow A\otimes A$ satisfies the following relations
 \begin{subequations}
 \begin{eqnarray}\label{eq:coalgebra1}
& (R_{\cdot}(x)\otimes \id-\sigma(R_{\cdot}(x)\otimes \id))\Delta(y)+
(\id\otimes L_{\cdot}(y)-\sigma(\id\otimes L_{\cdot}(y)))\Delta(x)=0,
 \\
 \label{eq:coalgebra2}
& \begin{array}{r}
(L_{\cdot}(x)\otimes \id)\Delta(y)+
 \sigma(L_{\cdot}(y)\otimes \id)\Delta(x)=\Delta(x\cdot y) \\
 = \sigma(\id\otimes R_{\cdot}(x))\Delta(y)
 +(\id\otimes R_{\cdot}(y))\Delta(x).
\end{array}
 \end{eqnarray}
 \end{subequations}
\end{thm}
\begin{proof}
For any $a, b\in A^*$ and any $x,y\in A$ we have
\begin{align*}
&\langle
(R_{\cdot}(x)\otimes \id)\Delta(y), a\otimes b
\rangle=
\langle
y, (R_{\cdot}^*(x)a)\circ b
\rangle=
\langle
L_{\circ}^*(R_{\cdot}^*(x)a)y, b
\rangle,\cr
&\langle
\sigma(R_{\cdot}(x)\otimes \id)\Delta(y), a\otimes b
\rangle=
\langle
y, (R_{\cdot}^*(x)b)\circ a
\rangle=
\langle
R_{\circ}^*(a)y
, R_{\cdot}^*(x)b
\rangle=
\langle
(R_{\circ}^*(a)y)\cdot x
, b
\rangle,\cr
&\langle
(\id\otimes L_{\cdot}(y))\Delta(x), a\otimes b
\rangle=
\langle
x, a\circ (L_{\cdot}^*(y)b)
\rangle=
\langle
y\cdot (L_{\circ}^*(a)x)
,b
\rangle,\cr
&\langle
\sigma(\id\otimes L_{\cdot}(y))\Delta(x), a\otimes b
\rangle=
\langle
x, b\circ (L_{\cdot}^*(y)a)
\rangle=
\langle
R_{\circ}^*(L_{\cdot}^*(y)a)x
,b
\rangle.
\end{align*}
Hence  \eqref{eq:matchedpair1} is equivalent to  \eqref{eq:coalgebra1}.

Similarly, we have for any $x,y\in A$ and
any $a, b\in A^*$
\begin{eqnarray*}
&&
\langle
\Delta(x\cdot y), a\otimes b
\rangle=
\langle
x\cdot y, a\circ b
\rangle=
\langle
L_{\circ}^*(a)(x\cdot y)
,b
\rangle=
\langle
R_{\circ}^*(b)(x\cdot y)
,a
\rangle,\cr
&&
\langle
(L_{\cdot}(x)\otimes \id)\Delta(y), a\otimes b
\rangle=
\langle
y, (L_{\cdot}^*(x)a)\circ b
\rangle=
\langle
L_{\circ}^*(L_{\cdot}^*(x)a)y
,b
\rangle,\cr
&&
\langle
\sigma(L_{\cdot}(y)\otimes \id)\Delta(x), a\otimes b
\rangle=
\langle
x, (L_{\cdot}^*(y)b)\circ a
\rangle=
\langle
y\cdot (R_{\circ}^*(a)x)
,b
\rangle,\cr
&&
\langle
 \sigma(\id\otimes R_{\cdot}(x))\Delta(y), a\otimes b
\rangle=
\langle
y,
b\circ (R_{\cdot}^*(x)a)
\rangle=
\langle
R_{\circ}^*(R_{\cdot}^*(x)a)y
,b
\rangle,\cr
&&
\langle
(\id\otimes R_{\cdot}(y))\Delta(x), a\otimes b
\rangle=
\langle
x, a\circ (R_{\cdot}^*(y)b)
\rangle=
\langle
(L_{\circ}^*(a)x)\cdot y
,b
\rangle,
\end{eqnarray*}
Therefore,  \eqref{eq:matchedpair2} and   \eqref{eq:matchedpair3}
and  is equivalent to  \eqref{eq:coalgebra2}.
\end{proof}
\begin{rmk}
Obviously, if $L$ and $R$ commute, then $L^*$ and $R^*$ commute too and
if in addition $\gamma: A^*\rightarrow A^*\otimes A^*$ is a linear maps such that
its dual $\gamma^*: A\otimes A\rightarrow A$ defines a nearly associative algebra
structure $\cdot$ on $A$, then $\Delta$ satisfies
 \eqref{eq:coalgebra1} and \eqref{eq:coalgebra2} if and only if $\gamma$
satisfies for all $a,b\in A^*$,
\begin{align*}
& (R_{\circ}(a)\otimes \id-\sigma(R_{\circ}(a)\otimes \id))\gamma(b)+
(\id\otimes L_{\circ}(b)-\sigma(\id\otimes L_{\circ}(b)))\gamma(a)=0,
\\
& (L_{\circ}(x)\otimes \id)\gamma(b)+
 \sigma(L_{\circ}(b)\otimes \id)\gamma(a)= \\
& \qquad \qquad \qquad
\gamma(a\circ b)= \sigma(\id\otimes R_{\circ}(a))\gamma(b)
 +(\id\otimes R_{\circ}(b))\gamma(a).
 \end{align*}
\end{rmk}
\begin{defi}
Let $(A, \cdot)$ be a nearly associative algebra in which the left ($L$) and
right ($R$) multiplication operators commute.
A nearly anti-flexible bialgebra structure is a linear map
$\Delta:A\rightarrow A\otimes A$ such that
\begin{itemize}
\item $\Delta^*:A^*\otimes A^*\rightarrow A^*$ defines
a nearly associative algebra structure on $A,$
\item $\Delta$ satisfies  \eqref{eq:coalgebra2} and \eqref{eq:coalgebra2}.
\end{itemize}
\end{defi}
\begin{thm}
Let $(A,\cdot)$ be a nearly associative algebra in which the left and right
multiplication operators commute.
Suppose that there is a nearly associative algebra structure
on $A^*$ denoted by $\circ$ which defined a linear map
$\Delta: A\rightarrow A\otimes A$. Then the  following conditions
are equivalent:
\begin{enumerate}[label=\upshape{(\roman*)},left=7pt]
\item
$(A\oplus A^*, A,A^*)$ is a standard  Manin triple of
nearly associative  algebras  $(A, \cdot)$ and $(A^*, \circ)$
such that its associated symmetric bilinear form  $\mathfrak{B}_d$ is
defined by  \eqref{eq:sbl}.
\item
$(A, A^*, R_{\cdot}^*,
L_{\cdot}^*, R_{\circ}^*, L_{\circ}^* )$ is a matched pair of
nearly associative algebras $(A, \cdot)$ and $(A^*, \circ)$.
\item $(A, A^*)$ is a nearly associative bialgebra.
\end{enumerate}
\end{thm}




\section{Hom-Lie admissible, G-Hom-associative,  flexible and anti-flexible Hom-algebras} \label{sec:homLieadmGHomass}
Hom-Lie admissible algebras along with Hom-associative algebras and more general $G$-Hom-associative algebras were first introduced, and Hom-associative algebras and $G$-Hom-associative algebras were shown to be Hom-Lie admissible in \cite{DassoundoSilvestrov:nearlyhomass:MakhloufSilvestrov:homstructure}.

Hom-algebra is a triple $(A, \mu, \alpha)$ consisting of a linear space $A$ over a field $\mathbb{K}$, a bilinear product $\mu:A\times A\rightarrow A$ and a linear map $\alpha: A\rightarrow A$.

\begin{defi}[\cite{DassoundoSilvestrov:nearlyhomass:MakhloufSilvestrov:homstructure}]
Hom-Lie, Hom-Lie admissible, Hom-associative and $G$-Hom-associative Hom-algebras (over a field $\mathbb{K}$) are defined as follows:
\begin{enumerate}[label=\upshape{\arabic*)},left=5pt]
\item Hom-Lie algebras are triples $( A, [.,.], \alpha)$,
consisting of a linear space $A$ over a field $\mathbb{K}$, bilinear map
{\rm(}bilinear product{\rm)}
$[.,.]: A\times A\rightarrow A$ and a linear map
$\alpha: A\rightarrow A$ satisfying, for all $x, y, z\in  A$,
\begin{align} \label{eq_hom_skewsym_identity}
[x,y]=-[y,x],  & \quad \quad \text{\rm (Skew-symmetry)}\\
\label{eq_hom_jacobi_identity}
[\alpha(x),[y,z]]+[\alpha(y), [z,x]]+ [\alpha(z), [x,y]] = 0. & \quad \quad \text{\rm (Hom-Jacobi identity)}
\end{align}
\item Hom-Lie admissible algebras are Hom-algebras $(A, \mu, \alpha)$ consisting of possibly non-associative algebra $(A, \mu)$ and a linear map $\alpha: A \rightarrow A$, such that
    $(A,[.,.],\alpha)$ is a Hom-Lie algebra, where $[x,y]=\mu(x,y)-\mu(y,x)$ for all $x,y\in  A$.
\item
Hom-associative algebras are triples $(A, \cdot, \alpha)$ consisting of
a linear space $A$ over a field $\mathbb{K}$, a bilinear product $\mu:A\times A\rightarrow A$
and a linear map $\alpha: A\rightarrow A$, satisfying for all $x,y,z\in A$,
\begin{equation}\label{Homass}
\mu(\mu(x, y),\alpha(z))=\mu(\alpha(x),\mu(y,z)). \quad \quad \text{\rm(Hom-associativity)}
\end{equation}
\item
Let $G$ be a subgroup of the permutations group $\mathcal{S}_3$.  Hom-algebra $(A, \mu, \alpha)$ is said to be $G$-Hom-associative if
\begin{equation}\label{admi}
\sum_{\sigma\in G} {(-1)^{\varepsilon ({\sigma})}
(\mu (\mu (x_{\sigma (1)},x_{\sigma (2)}),\alpha
(x_{\sigma (3)}))}-\mu(\alpha(x_{\sigma (1)}),\mu
(x_{\sigma (2)},x_{\sigma (3)}))=0,
\end{equation}
where $x_i \in A, i=1,2,3$ and $(-1)^{\varepsilon
({\sigma})}$ is the signature of the permutation
$\sigma$.
\end{enumerate}
\end{defi}
For any Hom-algebra $(A, \mu, \alpha)$, the Hom-associator, called also $\alpha$-associator of $\mu$, is a trilinear map (ternary product) $a_{\alpha,\mu}:A\times A \times A \rightarrow A$ defined by
\begin{equation*}
a_{\alpha,\mu}(x_1,x_2,x_3)=\mu (\mu
(x_1,x_2),\alpha (x_{3}))-\mu(\alpha(x_{1}),\mu
(x_{2},x_3)) \end{equation*}
for all $x_1,x_2,x_3 \in A$.
The ordinary associator
$$ a_{\mu}(x_1,x_2,x_3)=a_{\id,\mu}(x_1,x_2,x_3)
=\mu((x_1,x_2),(x_{3}))-\mu((x_{1}),\mu(x_{2},x_3))$$
on an algebra $(A,\mu)$ is $\alpha$-associator for the Hom-algebra $(A, \mu, \alpha)=(A,\mu,\id)$ with $\alpha=\id: A \rightarrow A$, the identity map on $A$.

Using Hom-associator $a_{\alpha,\mu}$ and notation $\sigma(x_1,x_2,x_3)=(x_{\sigma(1)},x_{\sigma(2)},x_{\sigma(3)})$, the Hom-associativity \eqref{Homass} can be written as
\begin{equation}\label{Homass_ass}
a_{\alpha,\mu}(x,y,z)=\mu(\mu(x, y),\alpha(z))-\mu(\alpha(x),\mu(y,z))=0, \quad \quad \text{\rm(Hom-associativity)}
\end{equation}
or as  $a_{\alpha,\mu}=0$,
and the $G$-Hom-associativity \eqref{admi} as
\begin{equation}
\sum_{\sigma\in G}{(-1)^{\varepsilon
({\sigma})}a_{\alpha,\mu}\circ \sigma}=0.
\end{equation}
If $\mu$ is the multiplication of a Hom-Lie
admissible Lie algebra, then
\eqref{admi} is equivalent to $[ x,y ]=\mu (x,y)-\mu (y,x )$
satisfying the Hom-Jacobi identity, or
equivalently,
\begin{equation}
\sum_{\sigma\in \mathcal{S}_3}
{(-1)^{\varepsilon({\sigma})} (\mu (\mu
(x_{\sigma (1)},x_{\sigma (2)}),\alpha (x_{\sigma
(3)}))}-\mu(\alpha(x_{\sigma (1)}),\mu (x_{\sigma
(2)},x_{\sigma (3)})))=0,
\end{equation}
which may be written as
\begin{equation}
\sum_{\sigma\in \mathcal{S}_3}{(-1)^{\varepsilon
({\sigma})}a_{\alpha,\mu}\circ \sigma}=0.
\end{equation}
Thus, Hom-Lie admissible Hom-algebras are $\mathcal{S}_3$-associative Hom-algebras.
In general, for all subgroups $G$ of the permutations group $\mathcal{S}_3$, all $G$-Hom-associative Hom-algebras are Hom-Lie admissible, or in other words, all Hom-algebras from the six classes of $G$-Hom-associative Hom-algebras, corresponding to the six subgroups of the symmetric group $\mathcal{S}_3$, are Hom-Lie admissible \cite[Proposition 3.4]{DassoundoSilvestrov:nearlyhomass:MakhloufSilvestrov:homstructure}.
All six subgroups of
$\mathcal{S}_3$ are
\begin{eqnarray*}
& G_1=\mathcal{S}_3(\id)=\{\id\},
G_2=\mathcal{S}_3(\tau_{12})=\{\id,\tau_{1 2}\},
G_3=\mathcal{S}_3(\tau_{23})=\{\id,\tau_{23}\}, &\\
& G_4=\mathcal{S}_3(\tau_{13})=\{\id,\tau_{1 3}\},
G_5=\mathcal{A}_3,
G_6=\mathcal{S}_3 &
\end{eqnarray*}
where $\mathcal{A}_3$
is the alternating group and $\tau_{ij}$ is
the transposition of $i$ and $j$.
{\tiny \begin{table}[ht!]
\caption{$G$-Hom-associative algebras}
\label{Table:GHomass}
\begin{center}
{\small
\begin{tabular}{|c|l|c|}
 \hline
\begin{minipage}{0.10\linewidth}
 \vspace{1mm} Subgroup \\
 of $\mathcal{S}_3$
 \vspace{1mm}
\end{minipage}
 & \begin{minipage}{0.22\linewidth}
 \vspace{1mm}
 \centering{
 Hom-algebras \\
 class names
 }
 \vspace{1mm}
\end{minipage} &
\begin{minipage}{0.30\linewidth}
 \vspace{2mm} \centering{
 Defining Identity \\ (Notation: $\mu(a,b)=ab$)
 }
 \vspace{2mm}
\end{minipage}
\\
 \hline
 \begin{minipage}{0.10\linewidth}
 \vspace{1mm}
$G_1=$\\
$\mathcal{S}_3(\id)$
\end{minipage}
&
\begin{minipage}{0.22\linewidth}
\vspace{2mm}
Hom-associative
\vspace{2mm}
\end{minipage}
& {\small $\alpha(x) (yz)=(xy)\alpha (z)$ } \\
\hline
\begin{minipage}{0.10\linewidth}
 $G_2=$ \\
 $\mathcal{S}_3(\tau_{12})$
\end{minipage}
 &
\begin{minipage}{0.23\linewidth}
\vspace{2mm}
Hom-left symmetric \\
Hom-Vinberg
\vspace{2mm}
\end{minipage}
&
{\small
$ \alpha(x)(yz)-\alpha(y)(xz)= (xy)\alpha(z)-(yx)\alpha (z) $} \\
\hline
 \begin{minipage}{0.10\linewidth}
 $G_3=$ \\
 $\mathcal{S}_3(\tau_{23})$
\end{minipage} &
\begin{minipage}{0.25\linewidth}
\vspace{2mm}
$\mathcal{S}_3(\tau_{23})$-Hom-associative\\
Hom-right symmetric \\
Hom-pre-Lie
\vspace{2mm}
\end{minipage} &
{\small $\alpha(x)(yz)-\alpha(x)(zy)=(xy)\alpha(z)-(xz)\alpha(y)$ } \\
\hline
\begin{minipage}{0.10\linewidth}
 $G_4=$ \\ $\mathcal{S}_3(\tau_{13})$
 \end{minipage} &
 \begin{minipage}{0.25\linewidth}
\vspace{2mm}
$\mathcal{S}_3(\tau_{13})$-Hom-associative\\
Hom-anti-flexible\\
Hom-center symmetric
\vspace{2mm}
\end{minipage}
   &
{\small
$ \alpha(x)(yz)-\alpha(z)(yx)= (xy)\alpha (z)- (zy)\alpha (x) $ } \\
 \hline
\begin{minipage}{0.10\linewidth}
 $G_5=$  $\mathcal{A}_3$
 \end{minipage}
  &
\begin{minipage}{0.23\linewidth}
\vspace{2mm}
$\mathcal{A}_3$-Hom-associative
\vspace{2mm}
\end{minipage} &
{\small
$\begin{array}{l}
\alpha(x)(yz)+\alpha(y)(zx)+\alpha(z)(xy)= \\
(xy)\alpha(z)+(yz)\alpha(x) +(zx)\alpha(y)
\end{array}$ }\\
 \hline
\begin{minipage}{0.10\linewidth}
 $G_6=$  $\mathcal{S}_3$
 \end{minipage}
 & \begin{minipage}{0.23\linewidth}
\vspace{2mm}
Hom-Lie admissible
\vspace{2mm}
\end{minipage}
 &
{\small $\begin{array}{r}
\displaystyle{\sum_{\sigma\in \mathcal{S}_3} (-1)^{\varepsilon({\sigma})}}
\left((x_{\sigma(1)}x_{\sigma(2)})\alpha(x_{\sigma (3)})
\right. \\
\left.
-\alpha(x_{\sigma (1)})(x_{\sigma(2)}x_{\sigma(3)})\right)=0
\end{array}$}
\\
\hline     
    \end{tabular}
    }
  \end{center}
\end{table}
}

The skew-symmetric $G_5$-Hom-associative Hom-algebras and Hom-Lie algebras form the same class of Hom-algebras for linear spaces over fields of characteristic different from $2$,
since then the defining identity of $G_5$-Hom-associative algebras is equivalent to the Hom-Jacobi identity of Hom-Lie algebras when the product $\mu$ is skew-symmetric.

A Hom-right symmetric (Hom-pre-Lie) algebra is the opposite algebra of a Hom-left-symmetric algebra.

Hom-flexible algebras introduced in \cite{DassoundoSilvestrov:nearlyhomass:MakhloufSilvestrov:homstructure} is a generalization to Hom-algebra context of  flexible algebras
\cite{DassoundoSilvestrov:nearlyhomass:Albert:PowerAssRings,
DassoundoSilvestrov:nearlyhomass:Myung:Lieadmalg,
DassoundoSilvestrov:nearlyhomass:Myung:LiealgFlexLieadmalg}.

\begin{defi}[\cite{DassoundoSilvestrov:nearlyhomass:MakhloufSilvestrov:homstructure}]
A Hom-algebra $(A, \mu, \alpha)$ is called
 flexible if
\begin{equation}\label{flexible}
      \mu (\mu (x,y), \alpha (x))=
      \mu (\alpha (x),\mu (y,x)))
\end{equation}
for any $x,y$ in $A$.
 \end{defi}
 Using the $\alpha$-associator
$ a_{\alpha,\mu}(x,y,z)=
 \mu(\mu(x,y),\alpha(z))-\mu(\alpha(x),\mu(y,z)),
$
  the condition \eqref{flexible} may be written as
 \begin{equation}\label{flexible2}
      a_{\alpha,\mu}(x,y,x)=0.
\end{equation}
Since Hom-associator map $a_{\alpha,\mu}$ is a trilinear map,
$$ a_{\alpha,\mu}(z-x,y,z-x)=a_{\alpha,\mu}(z,y,z)+a_{\alpha,\mu}(x,y,x)-a_{\alpha,\mu}(x,y,z)-a_{\alpha,\mu}(z,y,x),
$$
and hence \eqref{flexible2} yields
\begin{equation}\label{flexible3}
a_{\alpha,\mu}(x,y,z)=-a_{\alpha,\mu}(z,y,x)
\end{equation}
in linear spaces over any field, whereas setting $x=z$ in \eqref{flexible3} gives  $2a_{\alpha,\mu}(x,y,x)=0$, implying that \eqref{flexible2} and \eqref{flexible3} are equivalent in linear spaces over fields of characteristic different from $2$.
The equality \eqref{flexible3} written in terms of the Hom-algebra producs $\mu$ is
\begin{equation}\label{flexible4}
\mu(\mu(x,y),\alpha(z))- \mu(\alpha(x),\mu(y,z)) =\mu(\alpha(z),\mu(y,x))-\mu(\mu(z,y),\alpha(x)).
\end{equation}

\begin{defi}
\label{def:antiflexible}
A Hom-algebra $( A,\mu, \alpha)$ is called anti-flexible if
\begin{eqnarray} \label{antiflexible1}
\mu(\mu(x, y), \alpha(z))-\mu(\mu(z, y),\alpha(x))=
\mu(\mu(\alpha(x), \mu(y, z))-\mu(\mu(\alpha(z), \mu(y, x))
\end{eqnarray}
for all $x, y, z\in  A$.
\end{defi}

The equality \eqref{antiflexible1} can be written as
	\begin{eqnarray}\label{antiflexible3}
	a_{\alpha,\mu}(x,y,z)=a_{\alpha,\mu}(z, y, x),
	\end{eqnarray}
in terms of the Hom-associator $a_{\alpha,\mu}(x,y,z)$.

Hom-anti-flexible algebras were first introduced in \cite{DassoundoSilvestrov:nearlyhomass:MakhloufSilvestrov:homstructure} as $\mathcal{S}_3(\tau_{13})$-Hom-associative algebras, the subclass of $G$-Hom-associative algebras corresponding to the subgroup $G=\mathcal{S}_3(\tau_{13}) \subset \mathcal{S}_3$ (see Table \ref{Table:GHomass}).
In view of \eqref{antiflexible3}, anti-flexible algebras have been called Hom-center symmetric in \cite{DassoundoSilvestrov:nearlyhomass:HounkonnouDassoundo:homcensymalgbialg}.

Note that \eqref{antiflexible3} differs from \eqref{flexible3} by absence of the minus sign on the right hand side, meaning that for any $y$, the bilinear map
$a_{\alpha,\mu}(.,y,.)$ is symmetric on Hom-anti-flexible algebras and skew-symmetric on Hom-flexible algebras. Unlike \eqref{flexible} and \eqref{flexible3} in Hom-flexible algebras, in Hom-anti-flexible algebras, \eqref{antiflexible3} is generally not equivalent to the restriction of \eqref{antiflexible3} to $z=x$ trivially identically satisfied for any $x$ and $y$.
In view of \eqref{antiflexible3}, Hom-anti-flexible algebras are called Hom-center-symmetric algebras in \cite{DassoundoSilvestrov:nearlyhomass:HounkonnouDassoundo:homcensymalgbialg}.

\section{Nearly Hom-associative algebras, bimodules and matched pairs} \label{sec:nearlyhomass}
\begin{defi}
A nearly Hom-associative algebra is a triple $( A, \ast, \alpha)$, where
$ A$ is a linear space endowed to the bilinear product
$\ast: A\times A\rightarrow A$ and $\alpha: A\rightarrow A$
is a linear map such that for all $x,y,z\in  A$,
\begin{eqnarray}\label{eq_hom_identity}
\alpha(x)\ast(y\ast z)=(z\ast x)\ast \alpha(y).
\end{eqnarray}
\end{defi}

Nearly Hom-associative algebras are Hom-Lie admissible.
\begin{pro}\label{prop_hom_Lie_adm}
Any nearly Hom-associative algebra $( A, \ast, \alpha)$ is Hom-Lie admissible, that is
$( A, [.,.], \alpha)$ is a Hom-Lie algebra, where $[x,y]=x\ast y-y\ast x$ for all $x,y\in  A$.
\end{pro}
\begin{proof}
Let $( A, \ast, \alpha )$ be a nearly Hom-associative algebra. The commutator is skew-symmetric since $[x,y]=x\ast y-y\ast x=-(y\ast x-x\ast y)=-[y,x].$
For all $x,y,z\in  A$,
\begin{align*}{}
& [\alpha(x),[y,z]]+[\alpha(y), [z,x]]+[\alpha(z), [x,y]] \cr
& = [\alpha(x), y\ast z-z\ast y]+[\alpha(y), z\ast x-x\ast z]+[\alpha(z), x\ast y-y\ast x]\cr
& =\alpha(x)\ast (y\ast z)-\alpha(x)\ast(z\ast y)-(y\ast z)\ast \alpha(x)\cr
&  +(z\ast y)\ast \alpha(x)+\alpha(y)\ast (z\ast x)-\alpha(y)\ast (x\ast z)\cr
&  -(z\ast x)\ast \alpha(y)+(x\ast z)\ast \alpha(y)+\alpha(z)\ast(x\ast y)\cr
&  -\alpha(z)\ast (y\ast x)-(x\ast y)\ast \alpha(z)+(y\ast x)\ast \alpha(z)\cr
& =\{\alpha(x)\ast (y\ast z)-(z\ast x)\ast \alpha(y) \}\cr
&  +\{(y\ast x)\ast \alpha(z)-\alpha(x)\ast(z\ast y)\}\cr
&  +\{\alpha(y)\ast (z\ast x)-(x\ast y)\ast \alpha(z)\}\cr
&  +\{\alpha(z)\ast(x\ast y)-(y\ast z)\ast \alpha(x)\}\cr
&  +\{(z\ast y)\ast \alpha(x)-\alpha(y)\ast (x\ast z)\}\cr
&  +\{(x\ast z)\ast \alpha(y)-\alpha(z)\ast (y\ast x) \}=0.
\end{align*}
Therefore, $( A, [.,.], \alpha)$ is a Hom-Lie algebra.
\end{proof}

Commutative nearly Hom-associative algebras are Hom-anti-flexible.
\begin{pro}
	If $( A, \ast, \alpha)$ is a commutative  nearly Hom-associative  algebra,
	then $( A, \ast, \alpha)$ is a Hom-anti-flexible algebra.
\end{pro}
\begin{proof}
In a commutative nearly Hom-associative algebra $( A, \ast, \alpha)$.
\begin{align*}
a_{\alpha,\ast}(x,y,z) & = (x\ast y)\ast\alpha(z)- \alpha(x)\ast(y\ast z)&& \\
&= \alpha(y)\ast(z\ast x)-( z\ast x)\ast\alpha(y)  && \text{(nearly Hom-associativity)}\\
&= \alpha(y)\ast(x\ast z)-( x\ast z)\ast\alpha(y)   && \text{(commutativity)}\\
&=(z\ast y)\ast\alpha(x)- \alpha(z)\ast(y\ast x)  && \text{(nearly Hom-associativity)}\\
&=a_{\alpha,\ast}(z,y,x). &&
\end{align*}
So any commutative nearly Hom-associative   algebra is a Hom-anti-flexible  algebra.
\end{proof}

\begin{defi}
	A bimodule of a nearly Hom-associative   algebra $( A, \ast, \alpha)$ is a quadruple
	$(l, r, V, \varphi)$, where $V$  is a linear space,  $l,r: A\rightarrow \End(V)$ are two linear maps
	and $\varphi\in \End(V)$ satisfying the relations, for all $x,y\in  A$,
\begin{subequations}
	\begin{eqnarray}\label{eq_hom_bimodule0}
\varphi\circ l(x) = l({\alpha(x)})\circ \varphi, &&
	\varphi\circ r(x) = r({\alpha(x)})\circ \varphi,
	\\
\label{eq_hom_bimodule1}
	l({\alpha(x)})\circ l(y)&=&r({\alpha(y)})\circ r(x),
	\\
\label{eq_hom_bimodule2}
	l({\alpha(x)})\circ r(y)&=&l({y\ast x})\circ \varphi,
	\\
\label{eq_hom_bimodule3}
	r({\alpha(x)})\circ l(y)&=&r({x\ast y})\circ \varphi.
	\end{eqnarray}
\end{subequations}
\end{defi}
\begin{pro}\label{prop_hom_bimodule}
	Consider a nearly Hom-associative     $( A, \ast, \alpha)$. Let $l,r: A\rightarrow \End(V)$  be two linear maps such that $V$ is a linear space and $\varphi\in \End(V)$.
The quadruple $(l,r,V, \varphi)$ is a bimodule of
	$( A, \ast, \alpha)$ if and only if there is a structure of a nearly Hom-associative   algebra $\star$ on $ A\oplus V$ given by,
	for all $x,y\in  A$ and all  $u,v\in V$,
	\begin{eqnarray}
\label{eq_rpop_hom_bimodule1}
	(\alpha\oplus\varphi)(x+u)&=&\alpha(x)+\varphi(u),
	\\
\label{eq_prop_hom_bimodule2}
	(x+u)\star(y+v)&=&(x\ast y)+(l(x)v+r(y)u).
	\end{eqnarray}
\end{pro}
\begin{defi}\label{dfn_representation_hom_Lie}
	A representation of a Hom-Lie algebra $(\mathcal{G},[.,.]_{_\mathcal{G}}, \alpha_{_\mathcal{G}})$ on a linear space $V$ with respect
	to $\psi\in \End(V)$ is a linear map $\rho_{_\mathcal{G}}:\mathcal{G}\rightarrow \End(V)$ obeying for all $x, y\in \mathcal{G}$,
\begin{eqnarray}\label{eq_representation_hom_Lie1}
\rho_{_\mathcal{G}}(\alpha_{_\mathcal{G}}(x))\circ \psi &=& \psi\circ \rho_{_\mathcal{G}}(x), \\
\label{eq_representation_hom_Lie2}
	\rho_{_\mathcal{G}}([x,y]_{_\mathcal{G}})\circ\psi &=&\rho_{_\mathcal{G}}(\alpha_{_\mathcal{G}}(x))\circ \rho_{_\mathcal{G}}(y)-\rho_{_\mathcal{G}}(\alpha_{_\mathcal{G}}(y))\circ\rho_{_\mathcal{G}}(x).
	\end{eqnarray}
\end{defi}
\begin{pro}\label{propo_represention_Lie_sub_adjacent_hom_algebra}
	Let $(A, \cdot, \alpha)$ be a nearly Hom-associative   algebra and $V$ be a finite-dimensional linear
	space over the field $\mathbb{K}$ such that
	$(l, r,\varphi, V)$ is a bimodule of $( A, \cdot, \alpha)$, where $l, r:  A \rightarrow \End(V)$ are two linear maps and
	$\varphi\in \End(V)$. Then the linear map
	$l-r:  A \rightarrow \End(V), x \mapsto l(x)-r(x)$ is a  representation of the
	underlying Hom-Lie  algebra $( \mathcal{G}( A),\alpha )$ associated to the nearly Hom-associative algebra $( A, \cdot,\alpha)$.
\end{pro}
\begin{proof}
Let $( A, \cdot, \alpha)$ be a nearly Hom-associative algebra and $V$ a finite-dimensional linear
space over the field $\mathbb{K}$ such that
$(l, r,\varphi, V)$ is a bimodule of $( A, \cdot, \alpha)$, where $l, r:  A \rightarrow \End(V)$ are two linear maps and $\varphi\in \End(V)$. For all $x,y\in  A$,
\begin{align*}
& (l-r)({\alpha (x)})\circ \varphi=l({\alpha (x)})\circ \varphi-r({\alpha (x)})\circ \varphi
=\varphi\circ l(x)-\varphi\circ r(x)=\varphi\circ (l-r)(x),
\cr
&(l-r)({(\alpha (x))})\circ (l-r)(y)- (l-r)({(\alpha (y))})\circ (l-r)(x) \\
& \quad  = l({\alpha (x)})\circ l(y)-l({\alpha (x)})\circ r(y)
-r({\alpha (x)})\circ l(y)+r({\alpha (x)})\circ r(y) \cr
&
\quad \quad -l({\alpha (y)})\circ l(x)+l({\alpha (y)})\circ r(x)+r({\alpha (y)})\circ l(x)-r({\alpha (y)})\circ r(x)\cr
&
\quad=\{l({\alpha (x)})\circ l(y) -r({\alpha (y)})\circ r(x)  \}
 -l({\alpha (x)})\circ r(y)-r({\alpha (x)})\circ l(y)
\cr
&
\quad\quad +\{ r({\alpha (x)})\circ r(y)-l({\alpha (y)})\circ l(x) \}+
r({\alpha (y)})\circ l(x)+l({\alpha (y)})\circ r(x)\cr
&
\quad=r({\alpha (y)})\circ l(x)-l({\alpha (x)})\circ r(y)
+l({\alpha (y)})\circ r(x)-r({\alpha (x)})\circ l(y)
\cr
&
\quad =r({y\cdot x})\circ \varphi-l({y\cdot x})\circ \varphi
+l({x\cdot y})\circ \varphi-r({x\cdot y})\circ \varphi
=(l-r)({[x,y]})\circ\varphi.
\end{align*}
Therefore, \eqref{eq_representation_hom_Lie1} and  \eqref{eq_representation_hom_Lie2} are satisfied.
\end{proof}
\begin{defi}
	Let $\displaystyle (\mathcal{G}, [.,.]_{_\mathcal{G}}, \alpha_{_\mathcal{G}})$ and
	$\displaystyle (\mathcal{H}, [.,.]_{_\mathcal{H}}, \alpha_{_\mathcal{H}})$ be two Hom-Lie algebras. Let
	$\displaystyle \rho_{_\mathcal{H}}: \mathcal{H} \rightarrow \End(\mathcal{G})$ and
	$\displaystyle \mu_{_\mathcal{G}}: \mathcal{G}\rightarrow \End(\mathcal{H})$
	be two Hom-Lie algebra representations, and
	$\alpha_{_\mathcal{G}}: \mathcal{G}\rightarrow \mathcal{G}$ and
$\alpha_{_\mathcal{H}}: \mathcal{H} \rightarrow \mathcal{H}$ two
linear maps such that for all
	$x, y \in \mathcal{G}, a, b \in \mathcal{H},$
	\begin{subequations}
	\begin{eqnarray}\label{eq_matched_Hom-Lie_commu_1}
\begin{array}{lll}
		\mu_{_\mathcal{G}}(\alpha_{_\mathcal{G}}(x))\left[a, b\right]_{_\mathcal{H}}&=&
		\left[\mu_{_\mathcal{G}}(x)a, \alpha_{_\mathcal{H}}(b)\right]_{_\mathcal{H}}+
\left[\alpha_{_\mathcal{H}}(a), \mu_{_ \mathcal{G}}(x)b\right]_{_\mathcal{H}}\\
		&&-\mu_{_\mathcal{G}}(\rho_{_\mathcal{H}}(a)x)(\alpha_{_\mathcal{H}}(b))+
\mu_{_\mathcal{G}}(\rho_{_\mathcal{H}}(b)x)(\alpha_{_\mathcal{H}}(a)),
\end{array}
	\\
\begin{array}{lll}
\label{eq_matched_Hom-Lie_commu_2}
		\rho_{_\mathcal{H}}(\alpha_{_\mathcal{H}}(a))\left[x, y\right]_{_\mathcal{G}}
		&=&\left[\rho_{_\mathcal{H}}(a)x, \alpha_{_\mathcal{G}}(y)\right]_{_\mathcal{G}}
		+\left[\alpha_{_\mathcal{G}}(x), \rho_{_\mathcal{H}}(a)y\right]_{_\mathcal{G}}\\
		&&-\rho_{_\mathcal{H}}(\mu_{_\mathcal{G}}(x)a)(\alpha_{_\mathcal{G}}(y))+
\rho_{_\mathcal{H}}(\mu_{_\mathcal{G}}(y)a)(\alpha_{_\mathcal{G}}(x)).
\end{array}
	\end{eqnarray}
	\end{subequations}
	Then, $\displaystyle( \mathcal{G}, \mathcal{H}, \mu, \rho, \alpha_{_\mathcal{G}}, \alpha_{_\mathcal{H}})$ is called a matched pair of
	the Hom-Lie algebras
	$\displaystyle  {\mathcal{G}}$ and $\displaystyle  \mathcal{H}$,
	and  denoted by
$\displaystyle  \mathcal{H} \bowtie_{\mu_{_\mathcal{G}}}^{\rho_{_\mathcal{H}}} {\mathcal{G}}.$
	In this case,
$\displaystyle (\mathcal{G}\oplus \mathcal{H}, [.,.]_{_{\mathcal{G}\oplus \mathcal{H}}},
\alpha_{_\mathcal{G}}\oplus\alpha_{_\mathcal{H}})$ defines a Hom-Lie algebra, where
	\begin{eqnarray}
	[(x+a),  (y+b)]_{_{\mathcal{G}\oplus \mathcal{H}}}=[x, y]_{_\mathcal{G}}+\rho_{_\mathcal{H}}(a)y-\rho_{_\mathcal{H}}(b)x+[a, b]_{_\mathcal{H}}+\mu_{_\mathcal{G}}(x)b-\mu_{_\mathcal{G}}(y)a.
	\end{eqnarray}
\end{defi}
\begin{thm}
	Let $( A, \cdot, \alpha_{A})$ and $( B, \circ, \alpha_{ B})$ be two nearly Hom-associative   algebras.
	Suppose  there are linear maps $l_{ A}, r_{ A}: A\rightarrow \End( B)$
	and $l_{ B}, r_{ B}: B\rightarrow \End( A)$ such that $(l_{ A}, r_{ A}, B, \alpha_{ B})$ and
	$(l_{ B}, r_{ B}, A, \alpha_{ A})$ are  bimodules of the nearly Hom-associative   algebras $( A,\cdot,  \alpha_{ A})$ and $( B, \circ, \alpha_{ B})$, respectively and
	satisfying the following conditions for all $x, y\in  A$ and  $a, b \in  B:$
\begin{subequations}
	 \begin{eqnarray}\label{eq_Hom-matched_1}
	 \alpha_{ A}(x)\cdot (r_{ B}(a)y) +(r_{ B}(l_{ A}(y)a)\alpha_{ A}(x)-(l_{ B}(a)x)\cdot \alpha_{ A}(y)
	-l_{ B}(r_{ A}(x)a)\alpha_{ A}(y)=0, && \quad
	\\
\label{eq_Hom-matched_2}
	 \alpha_{ A}(x)\cdot (l_{ B}(a)y)+r_{ B}(r_{ A}(y)a)\alpha_{ A}(x)-r_{ B}(\alpha_{ B}(a))(y\cdot x)=0,&&
	\\
\label{eq_Hom-matched_3}
	l_{ B}(\alpha_{ B}(a))(x\cdot y)-(r_{ B}(a)y)\cdot\alpha_{ A}(x)-l_{ B}(l_{ A}(y)a)\alpha_{ A}(x)=0,&&
	\\
\label{eq_Hom-matched_4}
	\alpha_{ B}(a)\circ(r_{ A}(x)b)+r_{ A}(l_{ B}(b)x)\alpha_{ B}(a)
	-(l_{ A}(x)a)\circ \alpha_{ B}(b)-l_{ A}(r_{ B}(a)x)\alpha_{ B}(b)=0,&&
	\\
\label{eq_Hom-matched_5}
	\alpha_{ B}(a)\circ(l_{ A}(x)b)+r_{ A}(r_{ B}(b)x)\alpha_{ B}(a)
	-r_{ A}(\alpha_{ A}(x))(b\circ a)=0,&&
	\\
\label{eq_Hom-matched_6}
	l_{ A}(\alpha_{ A}(x))(b\circ a)-(r_{ A}(x)a)\circ \alpha_{ B}(b)
	-l_{ A}(l_{ B}(a)x)\alpha_{ B}(b)=0.&&
	\end{eqnarray}
\end{subequations}
	Then, there is  a bilinear product defined on $ A\oplus B$ for all $x,y\in  A$, and all
	$a,b\in  B$, by
	\begin{eqnarray}\label{eq_hom_csa}
	(x+a)\ast(y+b)=(x\cdot y+l_{ B}(a)y+r_{ B}(b)x)+(a\circ b+l_{ A}(x)b+r_{ A}(y)a)
	\end{eqnarray}
	such that $( A\oplus B, \ast, \alpha_{ A}\oplus\alpha_{ B})$ is a
	nearly Hom-associative   algebra.
\end{thm}
\begin{proof}

Let $( A, \cdot, \alpha_{ A})$, $( B, \circ, \alpha_{ B})$ be two nearly Hom-associative   algebras,
$(l_{ A}, r_{ A},  B, \alpha_{ B})$ a bimodule of $( A, \cdot, \alpha_{ A})$ and
$(l_{ B}, r_{ B},  A, \alpha_{ A})$ a bimodule of $( B, \circ, \alpha_{ B})$.
For all $x,y\in  A$ and all $a,b\in  B$,
\begin{align*}
&(\alpha_{ A}(x)+\alpha_{ B}(a))\ast ((y+b)\ast (z+c))\cr
&\quad =\{(\alpha_{ A}(x))\cdot (l_{ B}(b)z)+r_{ B}(r_{ A}(z)b)\cdot (\alpha_{ A}(x))\}\cr
&\quad\quad +\{(\alpha_{ A}(x))\cdot(r_{ B}(c)y)+r_{ B}(l_{ A}(y)c)\alpha_{ A}(x)\}\cr
&\quad\quad +(\alpha_{ A}(x))\cdot (y\cdot z)+
l_{ B}(\alpha_{ B}(a))(y\cdot z)+l_{ B}(\alpha_{ B}(a))(l_{ B}(b)z)\cr
&\quad\quad +l_{ B}(\alpha_{ B}(a))(r_{ B}(c)y)+
r_{ B}(b\circ c)(\alpha_{ A}(x))\cr
&\quad\quad +
\{(\alpha_{ B}(a))\circ  (l_{ A}(y)c) +     r_{ A}(r_{ B}(c)y)\alpha_{ B}(a)\}\cr
&\quad\quad +\{ (\alpha_{ B}(a))\circ (r_{ A}(z)b)  + r_{ A}(l_{ B}(b)z)\alpha_{ B}(a)    \}\cr
&\quad\quad +(\alpha_{ B}(a))\circ (b\circ c)+r_{ A}(y\cdot z)\alpha_{ B}(a)+
l_{ A}(\alpha_{ A}(x))(b\circ c)\cr
&\quad\quad +l_{ A}(\alpha_{ A}(x))(l_{ A}(y)c)   +l_{ A}(\alpha_{ A}(x))(r_{ A}(z)b);
\\
&((z+c)\ast (x+a)  )\ast (\alpha_{ A}(y)+\alpha_{ B}(b)) \cr
&\quad =
\{(l_{ B}(c)x)\cdot (\alpha_{ A}(y)) +l_{ B}(r_{ A}(x)c)\alpha_{ A}(y)\}\cr
&\quad\quad +\{l_{ B}(l_{ A}(z)a)\alpha_{ A}(y)+(l_{ B}(c)x)\cdot \alpha_{ A}(y) \}\cr
&\quad\quad +(z\cdot x)\cdot (\alpha_{ A}(y))+l_{ B}(c\circ a)(\alpha_{ A}(y))+
r_{ B}(\alpha_{ B}(b)(z\cdot x)\cr
&\quad\quad +r_{ B}(\alpha_{ B}(b)(l_{ B}(c)x)+
r_{ B}(\alpha_{ B}(b)(r_{ B}(a)z)\cr
&\quad\quad +
\{(l_{ A}(z)a)\circ (\alpha_{ B}(b))+l_{ A}(r_{ B}(a)z)\alpha_{ B}
\} \cr
&\quad\quad +
\{    (r_{ A}(x)c)\circ (\alpha_{ B}(b))+l_{ A}(l_{ B}(c)x)\alpha_{ B}(b) \}\cr
&\quad\quad +  (c\circ a)\circ (\alpha_{ B}(b))+
l_{ A}(z\cdot x)(\alpha_{ B}(b))+
r_{ A}(\alpha_{ A}(y))(c\circ a)\cr
&\quad\quad +r_{ A}(\alpha_{ A}(y))(l_{ A}(z)a)
+r_{ A}(\alpha_{ A}(y))(r_{ A}(x)c)
\end{align*}
Using  \eqref{eq_Hom-matched_1}~-~\eqref{eq_Hom-matched_6} and the fact
that $(l_{ A}, r_{ A},  B, \alpha_{ B})$ and $(l_{ B}, r_{ B},  A, \alpha_{ A})$ are bimodules of the nearly Hom-associative   algebras  $( A, \cdot, \alpha_{ A} )$ and $( B, \circ, \alpha_{ B} )$, respectively,
we obtain that $( A\oplus B, \ast, \alpha_{ A}\oplus\alpha_{ B})$ is a nearly associative  algebra.
\end{proof}
\begin{defi}
	A matched pair of the nearly Hom-associative   algebras
	$(A,\cdot,\alpha_{A})$ and $(B,\circ,\alpha_{B})$ is the high-tuple
	$(A,B,l_{A}, r_{A},\alpha_{B},l_{B}, r_{B},\alpha_{A})$, where
	$l_{A}, r_{A}: A\rightarrow \End(B)$
	and $l_{B}, r_{B}: B\rightarrow \End(A)$ are linear maps such that
$(l_{A},r_{A},B,\alpha_{B})$ and
	$(l_{B}, r_{B},  A, \alpha_{A})$ are  bimodules of the nearly Hom-associative algebras
$( A,\cdot,\alpha_{A})$ and $(B,\circ,\alpha_{B})$, respectively, and satisfying	\eqref{eq_Hom-matched_1}~-~\eqref{eq_Hom-matched_6}.
\end{defi}
\begin{cor}
	Let $(A,B,l_{A},r_{A},\alpha_{B},l_{B},r_{B},\alpha_{A})$ be a matched pair of
	the nearly Hom-associative   algebras
	$(A,\cdot,\alpha_{A})$ and $(B,\circ,\alpha_{B}).$ Then,
$(\mathcal{G}(A),\mathcal{G}(B), l_{A}-r_{A}, l_{B}-r_{B}, \alpha_{A},\alpha_{B})$
	is a matched pair of the  underlying  Hom-Lie algebras
	$\mathcal{G}(A)$ and  $\mathcal{G}(B)$ of the nearly Hom-associative algebras
$(A,\cdot,\alpha_{A})$ and $(B,\circ,\alpha_{B} )$.
\end{cor}
\begin{proof}

Let $(A,B,l_{A},r_{A},\alpha_{B},l_{B},r_{B},\alpha_{A})$ be a matched pair of nearly Hom-associative   algebras $(A,\cdot,\alpha_{A})$ and $(B,\circ,\alpha_{B})$.
In view of Proposition~\ref{propo_represention_Lie_sub_adjacent_hom_algebra},
the linear maps $l_{A}-r_{A}:  A\longrightarrow \End(B)$ and
$l_{B}-r_{B}: B\longrightarrow\End( A)$ are representations of the underlying Hom-Lie algebras $( \mathcal{G}(A), \alpha_{A})$ and $( \mathcal{G}(B), \alpha_{B})$, respectively.
Therefore, \eqref{eq_matched_Hom-Lie_commu_1} is equivalent to \eqref{eq_Hom-matched_1}~-~\eqref{eq_Hom-matched_3} and similarly, \eqref{eq_matched_Hom-Lie_commu_2} is equivalent to  \eqref{eq_Hom-matched_4}~-~\eqref{eq_Hom-matched_6}.
\end{proof}


\begin{thebibliography}{11}

\bibitem{DassoundoSilvestrov:nearlyhomass:AizawaSaito}
Aizawa, N., Sato, H.: $q$-deformation of the Virasoro algebra with central extension, Phys. Lett. B \textbf{256}, 185-190 (1991) (Hiroshima Univ. preprint, HUPD-9012 (1990))

\bibitem{DassoundoSilvestrov:nearlyhomass:Albert:PowerAssRings}
Albert, A. A.: Power associative rings, Trans. Amer. Math. Soc. \textbf{64}, 552-593 (1948)

\bibitem{DassoundoSilvestrov:nearlyhomass:AmmarEjbehiMakhlouf:homdeformation}
Ammar, F., Ejbehi, Z., Makhlouf, A., Cohomology and deformations of Hom-algebras,  J. Lie Theory \textbf{21}, no. 4, 813-836 (2011)

\bibitem{DassoundoSilvestrov:nearlyhomass:ArmakanSilvFarh:envalcolhomLieal}
Armakan, A., Silvestrov, S., Farhangdoost, M.: Enveloping algebras of color hom-Lie algebras, Turk. J. Math. \textbf{43}, 316-339 (2019). (arXiv:1709.06164[math.QA], (2017))

\bibitem{DassoundoSilvestrov:nearlyhomass:ArmakanSilvFarh:envalcerttypecolhomLieal}
Armakan, A., Silvestrov, S.: Enveloping algebras of certain types
of color Hom-Lie algebras, In: Silvestrov, S., Malyarenko, A., Ran\u{c}i\'{c}, M. (eds.), Algebraic Structures and Applications, Springer Proceedings in Mathematics and Statistics, vol. 317, Ch. 10, 257-284, Springer (2020)

\bibitem{DassoundoSilvestrov:nearlyhomass:Bai:LeftsymbialgYangBaxtereq}	
Bai, C.: Left-symmetric  bialgebras  and  an  analogue  of  the  classical  Yang-Baxter  equation.  Commun. Contemp. Math.  \textbf{10}(2), 221-260  (2008)

\bibitem{DassoundoSilvestrov:nearlyhomass:BaiC:doublconstrfrobalg}
Bai, C.: Double constructions of Frobenius algebras, Connes cocycle and their duality. J. Noncommut.
Geom. \textbf{4}, 475-530 (2010).

\bibitem{DassoundoSilvestrov:nearlyhomass:Bakayoko:LaplacehomLiequasibialg}
Bakayoko, I.: Laplacian of Hom-Lie quasi-bialgebras, International Journal of Algebra, \textbf{8} (15), 713-727 (2014)

\bibitem{DassoundoSilvestrov:nearlyhomass:Bakayoko:LmodcomodhomLiequasibialg}
Bakayoko, I.: $L$-modules, $L$-comodules and Hom-Lie quasi-bialgebras, African Diaspora Journal of Mathematics, \textbf{17} 49-64 (2014)

\bibitem{DassoundoSilvestrov:nearlyhomass:BakBan:bimodrotbaxt}
Bakayoko, I., Banagoura, M.: Bimodules and Rota-Baxter Relations. J. Appl. Mech. Eng. \textbf{4}(5) (2015)

\bibitem{DassoundoSilvestrov:nearlyhomass:BakyokoSilvestrov:MultiplicnHomLiecoloralg}
Bakayoko, I., Silvestrov, S.: Multiplicative $n$-Hom-Lie color algebras,
In: Silvestrov, S., Malyarenko, A., Ran\u{c}i\'{c}, M. (Eds.), Algebraic Structures and Applications, Springer Proceedings in Mathematics and Statistics \textbf{317}, Ch. 7, 159-187, Springer (2020). (arXiv:1912.10216[math.QA] (2019))

\bibitem{DassoundoSilvestrov:nearlyhomass:BakyokoSilvestrov:HomleftsymHomdendicolorYauTwi}
Bakayoko, I., Silvestrov, S.:
Hom-left-symmetric color dialgebras, Hom-tridendri\-form color algebras and Yau's twisting generalizations, Afrika Matematika (accepted in January 2020), arXiv:1912.01441[math.RA] (2019)

\bibitem{DassoundoSilvestrov:nearlyhomass:BenMakh:Hombiliform}
Benayadi, S., Makhlouf, A.: Hom-Lie algebras with symmetric invariant nondegenerate bilinear forms, J. Geom. Phys. \textbf{76}, 38-60 (2014)

\bibitem{DassoundoSilvestrov:nearlyhomass:BenAbdeljElhamdKaygorMakhl201920GenDernBiHomLiealg}
Ben Abdeljelil, A., Elhamdadi, M., Kaygorodov, I., Makhlouf, A.: Generalized Derivations of $n$-BiHom-Lie algebras, In: Silvestrov, S., Malyarenko, A., Ran\u{c}i\'{c}, M. (Eds.), Algebraic Structures and Applications,  Springer Proceedings in Mathematics and Statistics \textbf{317}, Ch. 4, 81-97, Springer (2020). (arXiv:1901.09750[math.RA] (2019))

\bibitem{DassoundoSilvestrov:nearlyhomass:CaenGoyv:MonHomHopf}
Caenepeel S., Goyvaerts I.: Monoidal Hom-Hopf Algebras,  Comm. Algebra \textbf{39}  no. 6, 2216-2240 (2011)

\bibitem{DassoundoSilvestrov:nearlyhomass:ChaiElinPop}
Chaichian, M., Ellinas, D., Popowicz, Z.: Quantum conformal algebra with central extension, Phys. Lett. B \textbf{248}, 95-99 (1990)

\bibitem{DassoundoSilvestrov:nearlyhomass:ChaiIsLukPopPresn}
Chaichian, M., Isaev, A. P., Lukierski, J., Popowic, Z., Pre\v{s}najder, P.: $q$-deformations of Virasoro algebra and conformal dimensions, Phys. Lett. B \textbf{262}(1), 32-38 (1991)

\bibitem{DassoundoSilvestrov:nearlyhomass:ChaiKuLuk}
Chaichian, M., Kulish, P., Lukierski, J.: $q$-deformed Jacobi identity, $q$-oscillators and $q$-deformed infinite-dimensional algebras, Phys. Lett. B \textbf{237}, 401-406 (1990)

\bibitem{DassoundoSilvestrov:nearlyhomass:ChaiPopPres}
Chaichian, M., Popowicz, Z., Pre\v{s}najder, P.: $q$-Virasoro algebra and its relation to the $q$-deformed KdV system, Phys. Lett. B \textbf{249}, 63-65 (1990)

\bibitem{DassoundoSilvestrov:nearlyhomass:CurtrZachos1}
Curtright, T. L., Zachos, C. K.: Deforming maps for quantum algebras, Phys. Lett. B \textbf{243}, 237-244 (1990)

\bibitem{DassoundoSilvestrov:nearlyhomass:DamKu}
Damaskinsky, E. V., Kulish, P. P.: Deformed oscillators and their applications (in Russian), Zap. Nauch. Semin. LOMI \textbf{189}, 37-74 (1991) (Engl. transl.: J. Sov. Math., \textbf{62}, 2963-2986 (1992))

\bibitem{DassoundoSilvestrov:nearlyhomass:DaskaloyannisGendefVir}
Daskaloyannis, C.: Generalized deformed Virasoro algebras, Modern Phys. Lett. A \textbf{7}(9), 809-816 (1992)

\bibitem{DassoundoSilvestrov:nearlyhomass:ElduqueMyung:MutatAlternAlg}
Elduque, A., Myung, H. C.: Mutations of Alternative Algebras, Mathematics and Its Applications \textbf{278}, Springer-Verlag (1994)

\bibitem{DassoundoSilvestrov:nearlyhomass:GrMakMenPan:Bihom1}
Graziani, G., Makhlouf, A., Menini, C., Panaite, F.: BiHom-Associative Algebras, BiHom-Lie Algebras and BiHom-Bialgebras, SIGMA \textbf{11}(086), 34 pp (2015)

\bibitem{DassoundoSilvestrov:nearlyhomass:GuoZhZheUEPBWHLieA}
Guo, L., Zhang, B., Zheng, S.:
Universal enveloping algebras and Poincare-Birkhoff-Witt theorem for involutive Hom-Lie algebras, J. Lie Theory, \textbf{28}(3), 735-756 (2018). (arXiv:1607.05973[math.QA], (2016))

\bibitem{DassoundoSilvestrov:nearlyhomass:HassanzadehShapiroSutlu:CyclichomolHomasal}
Hassanzadeh, M., Shapiro, I., S{\"u}tl{\"u}, S.: Cyclic homology for Hom-associative algebras, J. Geom. Phys. \textbf{98}, 40-56 (2015)

\bibitem{DassoundoSilvestrov:nearlyhomass:HaLaSil}
Hartwig, J. T., Larsson, D., Silvestrov, S. D.: Deformations of Lie algebras using $\sigma$-derivations, J. Algebra, \textbf{295},  314-361 (2006).
(Preprint in Mathematical Sciences 2003:32, LUTFMA-5036-2003, Centre for Mathematical Sciences, Department of Mathematics, Lund Institute of Technology, 52 pp. (2003))

\bibitem{DassoundoSilvestrov:nearlyhomass:He:stronghomassociativity}
Hellstr{\"o}m, L.: Strong Hom-associativity, In: Silvestrov, S., Malyarenko, A., Ran\u{c}i\'{c}, M. (eds.), Algebraic Structures and Applications, Springer Proceedings in Mathematics and Statistics, vol. 317, 317-337, Springer (2020)

\bibitem{DassoundoSilvestrov:nearlyhomass:HeMaSiUnAlHomAss}
Hellstr{\"o}m, L., Makhlouf, A., Silvestrov, S. D.: Universal Algebra Applied to Hom-Associative Algebras, and More, In: Makhlouf A., Paal E., Silvestrov S., Stolin A. (eds), Algebra, Geometry and Mathematical Physics. Springer Proceedings in Mathematics and Statistics, vol 85. Springer, Berlin, Heidelberg, 157-199 (2014)

\bibitem{DassoundoSilvestrov:nearlyhomass:HounkonnouDassoundo:centersymalgbialg}
Hounkonnou, M. N., Dassoundo M. L.: Center-symmetric Algebras and Bialgebras: Relevant Properties and Consequences. In: Kielanowski P., Ali S., Bieliavsky P., Odzijewicz A., Schlichenmaier M., Voronov T. (eds) Geometric Methods in Physics. Trends in Mathematics. 2016, pp. 281-293. Birkh{\"a}user, Cham (2016)

\bibitem{DassoundoSilvestrov:nearlyhomass:HounkonnouHoundedjiSilvestrov:DoubleconstrbiHomFrobalg}
Hounkonnou,  M. N., Houndedji, G. D., Silvestrov, S.: Double constructions of biHom-Frobenius algebras, arXiv:2008.06645 [math.QA] (2020)

\bibitem{DassoundoSilvestrov:nearlyhomass:HounkonnouDassoundo:homcensymalgbialg}
Hounkonnou, M. N., Dassoundo, M. L.: Hom-center-symmetric algebras and bialgebras.  arXiv:1801.06539.

\bibitem{DassoundoSilvestrov:nearlyhomass:Hu}
Hu, N.: $q$-Witt algebras, $q$-Lie algebras, $q$-holomorph structure and representations,  Algebra Colloq. \textbf{6}, no. 1, 51-70 (1999)

\bibitem{DassoundoSilvestrov:nearlyhomass:Kassel92}
Kassel, C.: Cyclic homology of differential operators, the virasoro algebra and a $q$-analogue, Comm. Math. Phys. 146 (2), 343-356 (1992)

\bibitem{DassoundoSilvestrov:nearlyhomass:kms:narygenBiHomLieBiHomassalgebras2020}
Kitouni, A., Makhlouf, A., Silvestrov, S.: On $n$-ary generalization of BiHom-Lie algebras and BiHom-associative algebras, In: Silvestrov, S., Malyarenko, A., RancicRan\u{c}i\'{c}, M. (Eds.), Algebraic Structures and Applications, Springer Proceedings in Mathematics and Statistics \textbf{317}, Ch 5, 99-126, Springer (2020)

\bibitem{DassoundoSilvestrov:nearlyhomass:LarssonSigSilvJGLTA2008}
Larsson, D., Sigurdsson, G., Silvestrov, S. D.:
Quasi-Lie deformations on the algebra $\mathbb{F}[t]/(t^N)$,
J. Gen. Lie Theory Appl. 2, 201-205 (2008)

\bibitem{DassoundoSilvestrov:nearlyhomass:LarssonSilvJA2005:QuasiHomLieCentExt2cocyid}
Larsson, D., Silvestrov, S. D.: Quasi-Hom-Lie algebras, Central Extensions and $2$-cocycle-like identities, J. Algebra \textbf{288}, 321-344 (2005).
(Preprints in Mathematical Sciences  2004:3, LUTFMA-5038-2004, Centre for Mathematical Sciences, Department of Mathematics, Lund Institute of Technology, Lund University, (2004))


\bibitem{DassoundoSilvestrov:nearlyhomass:LarssonSilvCM2005:QuasiLieAl}
Larsson, D., Silvestrov, S. D.: Quasi-Lie algebras, In: Noncommutative Geometry and Representation Theory in Mathematical Physics, Contemp. Math. \textbf{391}, Amer. Math. Soc., Providence, RI, 241-248 (2005).
(Preprints in Mathematical Sciences 2004:30, LUTFMA-5049-2004, Centre for Mathematical Sciences, Department of Mathematics, Lund Institute of Technology, Lund University (2004))

\bibitem{DassoundoSilvestrov:nearlyhomass:LSGradedquasiLiealg} Larsson, D., Silvestrov, S. D.:
Graded quasi-Lie agebras, Czechoslovak J. Phys. \textbf{55}, 1473-1478 (2005)

\bibitem{DassoundoSilvestrov:nearlyhomass:LarssonSilvestrov:Quasidefsl2twder}
Larsson, D., Silvestrov, S. D.: Quasi-deformations of $sl_2(\mathbb{F})$ using twisted derivations, Comm. Algebra \textbf{35}, 4303-4318 (2007)
(Preprint in Mathematical Sciences 2004:26, LUTFMA-5047-2004, Centre for Mathematical Sciences, Lund Institute of Technology, Lund University (2004). arXiv:math/0506172 [math.RA] (2005))


\bibitem{DassoundoSilvestrov:nearlyhomass:LiuKQuantumCentExt}
Liu, K. Q.: Quantum central extensions, C. R. Math. Rep. Acad. Sci. Canada \textbf{13} (4), 135-140 (1991)

\bibitem{DassoundoSilvestrov:nearlyhomass:LiuKQCharQuantWittAlg}
Liu, K. Q.: Characterizations of the Quantum Witt Algebra, Lett. Math. Phys. \textbf{24} (4), 257-265 (1992)

\bibitem{DassoundoSilvestrov:nearlyhomass:LiuKQPhDthesis} Liu, K. Q.: The Quantum Witt Algebra and Quantization of Some Modules over Witt Algebra, PhD Thesis, Department of Mathematics, University of Alberta, Edmonton, Canada (1992)

\bibitem{DassoundoSilvestrov:nearlyhomass:MaMakhSil:CurvedOoperatorSyst}
Ma, T., Makhlouf, A., Silvestrov, S.:
Curved $\mathcal{O}$-operator systems, 17pp, arXiv:1710.05232[math.RA] (2017)

\bibitem{DassoundoSilvestrov:nearlyhomass:MaMakhSil:RotaBaxbisyscovbialg}
Ma, T., Makhlouf, A., Silvestrov, S.:
Rota-Baxter bisystems and covariant bialgebras, 30 pp, arXiv:1710.05161[math.RA] (2017)

\bibitem{DassoundoSilvestrov:nearlyhomass:MaMakhSil:RotaBaxCosyCoquasitriMixBial}
Ma, T., Makhlouf, A., Silvestrov, S.:
Rota-Baxter Cosystems and Coquasitriangular Mixed Bialgebras, J. Algebra Appl. Accepted 2019. (Puiblished first online 2020: doi:10.1142/S021949882150064X)

\bibitem{DassoundoSilvestrov:nearlyhomass:Majid:matchedpeirsLiegrYangBa}
Majid, S.: Matched pairs of Lie groups associated to solutions of the Yang-Baxter equations, Pacific J. Math. \textbf{141}(2), 311-332  (1990)

\bibitem{DassoundoSilvestrov:nearlyhomass:MakhloufSilvestrov:homstructure}
Makhlouf, A., Silvestrov, S. D.: Hom-algebra structures,
J. Gen. Lie Theory Appl. \textbf{2}(2), 51-64 (2008).
(Preprints in Mathematical Sciences  2006:10, LUTFMA-5074-2006, Centre for Mathematical Sciences, Department of Mathematics, Lund Institute of Technology, Lund University (2006))

\bibitem{DassoundoSilvestrov:nearlyhomass:MakhSil:HomHopf}
Makhlouf, A., Silvestrov, S.:
Hom-Lie admissible Hom-coalgebras and Hom-Hopf algebras, In: Silvestrov, S., Paal, E., Abramov, V., Stolin, A. (Eds.), Generalized Lie Theory in Mathematics, Physics and Beyond, Ch. 17, 189-206, Springer-Verlag, Berlin, Heidelberg (2009). (arXiv:0709.2413[math.RA] (2007))

\bibitem{DassoundoSilvestrov:nearlyhomass:MakhSilv:HomDeform}
Makhlouf, A., Silvestrov, S. D.:
Notes on formal deformations of Hom-associative and Hom-Lie algebras, Forum Math. \textbf{22} (4), 715-739 (2010).
(arXiv:0712.3130 [math.RA] (2007))

\bibitem{DassoundoSilvestrov:nearlyhomass:MakhSilv:HomAlgHomCoalg}
Makhlouf, A., Silvestrov, S. D.:
Hom-Algebras and Hom-Coalgebras, J. Algebra Appl. \textbf{9}(04), 553-589 (2010). (arXiv:0811.0400[math.RA] (2008))

\bibitem{DassoundoSilvestrov:nearlyhomass:MakhloufYau:RotaBaHomLieadm}
Makhlouf,  A.,  Yau,  D.:  Rota-Baxter  Hom-Lie  admissible  algebras, Comm. algebra. \textbf{ 42}, 1231-1257 (2014)

\bibitem{DassoundoSilvestrov:nearlyhomass:Myung:Lieadmalg}
Myung, H. C.: Lie-admissible algebras, Hadronic J. \textbf{1}, 169-193 (1978)

\bibitem{DassoundoSilvestrov:nearlyhomass:MyungOkuboSantilli:ApplLieadmalgPhys}
Myung, H. C., Okubo, S., Santilli, R. M.: Applications of  Lie-admissible algebras in physics, Vol I, II, Hadronic Press (1978)

\bibitem{DassoundoSilvestrov:nearlyhomass:Myung:LiealgFlexLieadmalg}
Myung, H. C.: Lie algebras and flexible Lie-admissible algebras, Hadronic Press Monographs in
Mathematics \textbf{1}, Hadronic Press (1982)

\bibitem{DassoundoSilvestrov:nearlyhomass:RichardSilvestrovJA2008}
Richard, L., Silvestrov, S. D.: Quasi-Lie structure of $\sigma$-derivations of $\mathbb{C}[t^{\pm1}]$,
J. Algebra \textbf{319}(3), 1285-1304 (2008) (arXiv:math/0608196[math.QA] (2006). Preprints in mathematical sciences (2006:12), LUTFMA-5076-2006, Centre for Mathematical Sciences, Lund University (2006))

\bibitem{DassoundoSilvestrov:nearlyhomass:RichardSilvestrovGLTbnd20092009}
Richard, L., Silvestrov, S. D.:
A note on quasi-Lie and Hom-Lie structures of $\sigma$-derivations of ${\mathbb C}[z_1^{\pm 1},\ldots,z_n^{\pm 1}]$, In: Silvestrov, S., Paal, E., Abramov, V., Stolin, A. (Eds.), Generalized Lie Theory in Mathematics, Physics and Beyond, Ch. 22, 257-262, Springer-Verlag, Berlin, Heidelberg  (2009)

\bibitem{DassoundoSilvestrov:nearlyhomass:Santilli:introLieadmalg}
Santilli, R. M.: An  introduction  to  Lie-admissible  algebras, Nuovo Cim. Suppl., I. \textbf{6}, 1225-1249 (1968)

\bibitem{DassoundoSilvestrov:nearlyhomass:Santilli:Lieadmapprhadronicstr}
Santilli, R. M.: Lie-admissible approach to the hadronic structure, II , Hadronic Press (1982)

\bibitem{DassoundoSilvestrov:nearlyhomass:Schafer:intrononasal}
Schafer, R. D.: An  Introduction  to  Non-associative  Algebras, Dover  Publications (1961)

\bibitem{DassoundoSilvestrov:nearlyhomass:Sheng:homrep}
Sheng, Y.: Representations of Hom-Lie algebras, Algebr. Reprensent. Theory \textbf{15}, 1081-1098 (2012)

\bibitem{DassoundoSilvestrov:nearlyhomass:ShengBai:homLiebialg}
Sheng, Y.,  Bai,  C.: A  new  approach  to  Hom-Lie  bialgebras,  J. Algebra.  \textbf{399},  232-250  (2014)

\bibitem{DassoundoSilvestrov:nearlyhomass:SigSilGrquasiLieWitttype:CzJPhys2006}
Sigurdsson, G., Silvestrov, S.: Graded quasi-Lie algebras of Witt type, Czech. J. Phys. \textbf{56}, 1287-1291 (2006)

\bibitem{DassoundoSilvestrov:nearlyhomass:SigSilv:LiecolHomLieWittcex:GLTbdSpr2009}
Sigurdsson, G., Silvestrov, S.: Lie color and Hom-Lie algebras of Witt type and their central extensions, In: Silvestrov, S., Paal, E., Abramov, V., Stolin, A. (Eds.), Generalized Lie Theory in Mathematics, Physics and Beyond, Ch. 21, 247-255, Springer, Berlin, Heidelberg (2009)

\bibitem{DassoundoSilvestrov:nearlyhomass:SilvestrovParadigmQLieQhomLie2007}
Silvestrov, S.: Paradigm of quasi-Lie and quasi-Hom-Lie algebras and quasi-deformations. In "New techniques in Hopf algebras and graded ring theory", K. Vlaam. Acad. Belgie Wet. Kunsten (KVAB), Brussels, 165-177 (2007)

\bibitem{DassoundoSilvestrov:nearlyhomass:Yau:ModuleHomalg}
Yau, D.: Module Hom-algebras, arXiv:0812.4695[math.RA] (2008)

\bibitem{DassoundoSilvestrov:nearlyhomass:Yau:HomEnv}
Yau, D.: Enveloping algebras of Hom-Lie algebras, J. Gen. Lie Theory Appl. \textbf{2}(2), 95-108 (2008). (arXiv:0709.0849 [math.RA] (2007))

\bibitem{DassoundoSilvestrov:nearlyhomass:Yau:HomHom}
Yau, D.: Hom-algebras and homology, J. Lie Theory \textbf{19}(2), 409-421 (2009)

\bibitem{DassoundoSilvestrov:nearlyhomass:Yau:HombialgcomoduleHomalg}
Yau, D.: Hom-bialgebras and comodule Hom-algebras, \textit{Int. Electron.~J. Algebra} \textbf{8}, 45-64  (2010). (arXiv:0810.4866[math.RA] (2008))

\bibitem{DassoundoSilvestrov:nearlyhomass:Yau:HomYangBaHomLiequasitribial}
Yau  D.:  The  Hom-Yang-Baxter  equation,  Hom-Lie  algebras  and  quasi-triangular  bialgebras,  J. Phys. A.  \textbf{42}, 165--202  (2006)

\end{thebibliography}
\end{document}